\documentclass[12pt,twoside]{amsart}

\usepackage{amssymb}
\usepackage{graphicx}
\usepackage[pdfauthor={David Cook II, Uwe Nagel},
            pdftitle={The weak Lefschetz property for monomial ideals of small type},
            pdfsubject={Commutative algebra and enumerative combinatorics},
            pdfkeywords={Monomial ideals, weak Lefschetz property, determinants, permanents, lozenge tilings, non-intersecting lattice paths, perfect matchings},
            pdfproducer={LaTeX with hyperref},
            pdfcreator={latex->dvips->ps2pdf},
            pdfpagemode=UseOutlines,
            bookmarksopen=false,
            letterpaper,
            bookmarksnumbered=true,
            pdfborder={0 0 0},
            colorlinks=false]{hyperref}
\usepackage[margin=1in]{geometry}

\newcommand{\dntri}{\bigtriangledown}
\newcommand{\uptri}{\triangle}
\newcommand{\HF}{\mathcal{H}}
\newcommand{\PP}{\mathbb{P}}
\newcommand{\ZZ}{\mathbb{Z}}
\newcommand{\PS}{\mathfrak{S}}
\DeclareMathOperator{\charf}{char}
\DeclareMathOperator{\Mac}{Mac}
\DeclareMathOperator{\per}{perm}
\DeclareMathOperator{\rank}{rank}
\DeclareMathOperator{\sgn}{sgn}
\DeclareMathOperator{\msgn}{msgn}
\DeclareMathOperator{\lpsgn}{lpsgn}
\DeclareMathOperator{\soc}{soc}
\DeclareMathOperator{\Hom}{Hom}

\newcommand{\st}{\; \mid \;}
\def\urltilda{\kern -.15em\lower .7ex\hbox{\~{}}\kern .04em}
\newcommand{\flfr}[2]{\left\lfloor\frac{#1}{#2}\right\rfloor}

\numberwithin{figure}{section}
\numberwithin{equation}{section}

\newtheorem{theorem}{Theorem}[section]
\newtheorem{lemma}[theorem]{Lemma}
\newtheorem{proposition}[theorem]{Proposition}
\newtheorem{corollary}[theorem]{Corollary}
\newtheorem{conjecture}[theorem]{Conjecture}

\theoremstyle{definition}
\newtheorem{definition}[theorem]{Definition}
\newtheorem{remark}[theorem]{Remark}
\newtheorem{example}[theorem]{Example}

\begin{document}

\title{The weak Lefschetz property for monomial ideals of small type}

\author[D.\ Cook II]{David Cook II${}^{\star}$}
\address{Department of Mathematics \& Computer Science, Eastern Illinois University, Charleston, IL 46616}
\email{\href{mailto:dwcook@eiu.edu}{dwcook@eiu.edu}}

\author[U.\ Nagel]{Uwe Nagel}
\address{Department of Mathematics, University of Kentucky, 715 Patterson Office Tower, Lexington, KY 40506-0027}
\email{\href{mailto:uwe.nagel@uky.edu}{uwe.nagel@uky.edu}}

\thanks{
    Part of the work for this paper was done while the authors were partially supported by the National Security Agency
    under Grant Number H98230-09-1-0032.
    The second author was also partially supported by the National Security Agency under Grant Number H98230-12-1-0247
    and by the Simons Foundation under grants \#208869 and \#317096.\\
    \indent ${}^{\star}$ Corresponding author.}


\begin{abstract}
    In this work a combinatorial approach towards the weak Lefschetz property is developed that relates this property to enumerations
    of signed perfect matchings as well as to  enumerations of signed families of non-intersecting lattice paths in certain triangular regions.
    This connection is used to study Artinian quotients by monomial ideals of a three-dimensional polynomial ring.
Extending a main result in the recent memoir \cite{BMMNZ}, we completely classify the quotients of type two that have      the weak Lefschetz property in characteristic zero.
We also derive results in positive characteristic for quotients whose type is at most two.
\end{abstract}

\maketitle


\section{Introduction} \label{sec:intro}

A standard graded Artinian algebra $A$ over a field $K$ is said to have the \emph{weak Lefschetz property}
if there is a linear form $\ell \in A$ such that the multiplication map $\times \ell : [A]_i \rightarrow [A]_{i+1}$ has maximal rank
for all $i$ (i.e., it is injective or surjective).  The algebra $A$ has the  \emph{strong Lefschetz property} if
$\times \ell^d : [A]_i \rightarrow [A]_{i+d}$ has maximal rank for all $i$ and $d$. The names are motivated by the conclusion of
the Hard Lefschetz Theorem on the cohomology ring of a compact K\"ahler manifold. Many algebras are expected to have the Lefschetz
properties. However, deciding this problem is often very challenging.

The presence of the weak Lefschetz property has profound consequences for an algebra (see \cite{HMNW}). For example, Stanley used this in his contribution \cite{St-faces} towards the proof of the so-called $g$-Theorem that characterizes the face vectors of simplicial polytopes.
It has been a longstanding conjecture whether this characterization extends to the  face vectors of all triangulations of a sphere. In fact, this would be one of the consequences if one can show the so-called algebraic $g$-Conjecture,  which posits that a certain algebra has the strong Lefschetz property  (see \cite{NS1} and \cite{NS2}).  Although there has been a flurry of papers studying the Lefschetz properties
in the last decade (see, e.g.,  \cite{BMMNZ2, BK, BK-p, CGJL, GIV, HSS, HMMNWW, KRV, KV, LZ, MMO, MMN-2012}), we currently seem far
from being able to decide the above conjectures. Indeed, the need for new methods has led us to consider lozenge tilings, perfect
matchings, and families of non-intersecting lattice paths. We use this approach to establish new results about the presence or the
absence of the weak Lefschetz property for quotients of a polynomial ring $R$  in three variables. This is the first open case as any Artinian quotient of  a polynomial ring in two variables has even the strong Lefschetz property in characteristic zero \cite{HMNW}.

If $I$ is a monomial ideal, then $R/I$ is Artinian of type one if and only if  $I$ is generated by the  powers of the three variables. It is well-known that in this case  $R/I$ has the  Lefschetz
properties if the base field has characteristic zero (see \cite{Stanley-1980, ikeda, Wa, BTK}).   We extent this result by providing a version for base fields of  arbitrary characteristic
(see Theorem~\ref{thm:type-one}).

Monomial algebras $R/I$ of type two were considered in the recent memoir  \cite{BMMNZ}.
One of its main results  says that, in characteristic zero,  these algebras have the
weak Lefschetz property, provided they are also level. Examples show that this may fail if one drops the level assumption or if $K$ has positive characteristic. However, the intricate proof in the level case of \cite[Theorem 6.2]{BMMNZ} does not give any insight when
such failures occur. We resolve this by completely classifying all type two algebras that have the weak Lefschetz property if
the characteristic is zero or large enough (see Theorem~\ref{thm:type-two} and Proposition~\ref{pro:char-0-to-p}).

The structure of this paper is as follows.  In Section~\ref{sec:wlp}, we recall or derive some general results about
the presence of the weak Lefschetz property.  In Section~\ref{sec:tri}, we describe a key relation between a  monomial ideal and a triangular region, a certain planar region, as introduced in \cite{CN-resolutions}.  In Section~\ref{sec:enum}, we consider signed lozenge tilings of a triangular region, using two a priori different signs. We show that both signs lead to enumerations of signed lozenge tilings that completely control the presence of the weak Lefschetz property, regardless of the characteristic of the base field. In Section~\ref{sec:enumeration} we work out some enumerations explicitly. We describe some rather general techniques and then illustrate them by evaluation certain determinants. Sections ~\ref{sec:wlp} to \ref{sec:enumeration} constitute our combinatorial approach towards the Lefschetz properties of monomial ideals. We then apply it to study monomial algebras of type one and two in Section~\ref{sec:type1} and Section~\ref{sec:type2}, respectively.

\section{The weak Lefschetz property} \label{sec:wlp}

There are some general results that are helpful in order to determine the presence or absence of the weak Lefschetz
property. We recall or derive these tools here.

Throughout this paper, let $R = K[x_1,\ldots,x_n]$ be a standard graded polynomial ring in $n$ variables over a field $K$. Furthermore, all $R$-modules are assumed to be finitely generated and graded. The \emph{Hilbert function} of  an $R$-module $M$ is the function $h_M: \ZZ \to \ZZ$ defined by $h_M (j) = \dim_K [M]_j$.
 Let $M$ be an Artinian $R$-module. The
\emph{socle} of $M$, denoted $\soc{M}$, is the annihilator of $\mathfrak{m} = (x_1, \ldots, x_n)$, the homogeneous
maximal ideal of $R$, that is, $\soc M = \{y \in M \st y \cdot \mathfrak{m} = 0\}$.  The \emph{type} of $M$ is the
is the $K$-dimension of $\soc{M}$. The \emph{socle degree} or \emph{Castelnuovo-Mumford regularity} of $M$ is the maximum
degree of a non-zero element in $\soc{M}$. The module $M$ is said to be \emph{level} if all socle generators have the same
degree, i.e., its socle is concentrated in one degree.

Alternatively, assume that the minimal free resolution of $M$ over $R$ ends with a free module $\bigoplus_{i=1}^{m}
R(-t_i),$ where $t_1 \le \cdots \le t_m$  for all $i$. Then $M$ has $m$ socle generators with
degrees $t_1 - n, \ldots, t_m - n$. Thus, $M$ is level if and only if $m = 1$.
\smallskip

It was observed in \cite[Proposition~2.1(a)]{MMN-2011} that once multiplication by a general linear form on a quotient of $R$ is surjective, then it remains surjective. This can be extended to modules.

\begin{lemma} \label{lem:mod-surj}
    Let $M$ be a graded $R$-module such that the degrees of its minimal generators are at most $d$. Let $\ell \in R$ be
    a linear form. If the map $\times\ell: [M]_{d-1} \rightarrow [M]_{d}$ is surjective, then the map
    $\times\ell: [M]_{j-1} \rightarrow [M]_{j}$ is surjective for all $j \geq d$.
\end{lemma}

\begin{proof}
    Consider the exact sequence $[M]_{d-1} \stackrel{\times\ell}{\longrightarrow} [M]_{d} \rightarrow [M/\ell M]_{d} \rightarrow 0.$
    Notice the first map is surjective if and only if $[M/\ell M]_{d} = 0.$ Thus, the assumption gives $[M/\ell M]_{d} = 0$.
    Hence $[M/\ell M]_{j+1}$ is zero for all $j \geq d$ because $M$ does not have minimal generators having a degree
    greater than $d$, by assumption.
\end{proof}

As a consequence, we note a slight  generalization of \cite[Proposition~2.1(b)]{MMN-2011}, which considers the case of level algebras.

\begin{corollary}\label{cor:inj}
    Let $M$ be an Artinian graded $R$-module such that the degrees of its non-trivial socle elements are at least
    $\geq d-1$. Let $\ell \in R$ be a linear form. If the map $\times \ell: [M]_{d-1} \rightarrow [M]_{d}$ is
    injective, then the map $\times \ell: [M]_{j-1} \rightarrow [M]_{j}$ is injective for all $j \leq d$.
\end{corollary}

\begin{proof}
    Recall that the $K$-dual of $M$ is $M^{\vee} = \Hom_K (M, K)$. Then $\times \ell: [M]_{j-1} \rightarrow [M]_{j}$ is injective if
    and only if the map $\times \ell: [M^{\vee}]_{-j} \rightarrow [M^{\vee}]_{-j+1}$ is surjective. The assumption on
    the socle of $M$ means that the degrees of the minimal generators of $M^{\vee}$ are at most $-d+1$. Thus, we
    conclude by Lemma~\ref{lem:mod-surj}.
\end{proof}

The above observations imply that to decide the presence of the weak Lefschetz property we need only check near a
``peak'' of the Hilbert function.

\begin{proposition} \label{pro:wlp}
    Let $A \neq 0$ be an Artinian standard graded $K$-algebra. Let $\ell$ be a general linear form. Then there are the following facts:
    \begin{enumerate}
        \item Let $d$ be the smallest integer such that $h_{A}(d-1) > h_{A}(d)$.
        If $A$ has a non-zero socle element of degree less than $d-1$, then $A$ does not have  the weak Lefschetz property.

        \item Let $d$ be the largest integer such that $h_{A}(d-2) < h_{A}(d-1)$. If $A$ has the weak Lefschetz property, then
        \begin{enumerate}
            \item $\times \ell: [A]_{d-2} \rightarrow [A]_{d-1}$ is injective,
            \item $\times \ell: [A]_{d-1} \rightarrow [A]_{d}$ is surjective, and
            \item $A$ has no socle generators of degree less than $d-1$.
        \end{enumerate}

        \item Let $d \geq 0$ be an integer such that $A$ has the following three properties:
        \begin{enumerate}
            \item $\times \ell: [A]_{d-2} \rightarrow [A]_{d-1}$ is injective,
            \item $\times \ell: [A]_{d-1} \rightarrow [A]_{d}$ is surjective, and
            \item $A$ has no socle generators of degree less than $d-2$.
        \end{enumerate}
        Then $A$ has the weak Lefschetz property.
         \end{enumerate}
\end{proposition}
\begin{proof}
    Suppose in case (i) $A$ has a socle element $y \neq 0$ of degree $e < d-1$. Then $\ell y = 0$, and so the map
    $\times \ell: [A]_{e} \rightarrow [A]_{e+1}$ is not injective. Moreover, since $e < d-1$ we have $h_{R/I}(e) \le
    h_{R/I}(e+1)$. Hence, the map $\times \ell: [A]_{e} \rightarrow [A]_{e+1}$ does not have maximal rank. This proves
    claim (i).

    For showing (ii), suppose $A$ has the weak Lefschetz property. Then, by its definition, $A$ satisfies (ii)(a) and
    (ii)(b) because $h_A (d-1) \geq h_A (d)$. Assume (ii)(c) is not true, that is, $A$ has a socle element $y \neq 0$ of
    degree $e < d-1$. Then the map $\times \ell: [A]_{e} \rightarrow [A]_{e+1}$ is not injective. Since $A$ has the weak
    Lefschetz property, this implies $h_A (e) > h_A (e+1)$. Hence the assumption on $d$ gives $e \leq d-3$. However, this
    means that the Hilbert function of $A$ is not unimodal. This is impossible if $A$ has the weak Lefschetz property
    (see \cite{HMNW}).

    Finally, we prove (iii). Corollary~\ref{cor:inj} and Assumptions (iii)(a), and (iii)(c) imply that the map
    $\times \ell: [A]_{i-2} \rightarrow [A]_{i-1}$ is injective if $i \leq d$. Furthermore, using (iii)(b) and Lemma~\ref{lem:mod-surj}, we see that $\times \ell: [A]_{i-1} \rightarrow [A]_{i}$ is surjective if $i \geq d$.
    Thus, $A$ has the weak Lefschetz property.
\end{proof}

If the Hilbert function has two peaks in consecutive degrees, a situation to which we refer as ``twin peaks,'' the above arguments  give the following result.

\begin{corollary}\label{cor:twin-peaks-wlp}
    Let $A$ be an Artinian standard graded $K$-algebra, and let $\ell$ be a general linear form. Suppose there is an
    integer $d$ such that $0 \neq h_{A}(d-1) = h_{A}(d)$ and $A$ has no socle elements of degree less than $d-1$. Then
    $A$ has the weak Lefschetz property if and only if $\times \ell: [A]_{d-1} \rightarrow [A]_{d}$ is bijective.
\end{corollary}

The following easy, but useful observation is essentially the content of \cite[Proposition~2.2]{MMN-2011}.

\begin{proposition} \label{pro:mono}
    Let $A = R/I$ be an Artinian  $K$-algebra, where $I$ is generated by monomials and $K$ is an infinite field.  Let $d$ and $e>0$ be integers. Then the following conditions are equivalent:
    \begin{enumerate}
        \item The multiplication map $\times \ell^e: [A]_{d-e} \to [A]_d$ has maximal rank, where $\ell \in R$ is a general linear form.
        \item The multiplication map $\times (x_1 + \cdots + x_n)^e: [A]_{d-e} \to [A]_d$ has maximal rank.
    \end{enumerate}
\end{proposition}
\begin{proof}
    For the convenience of the reader we recall the argument. Let $\ell = a_1 x_1 + á á á + a_r x_r \in R$ be a general
    linear form. Thus, we may assume that each coefficient $a_i$ is not zero. Rescaling the variables $x_i$ such that
    $\ell$ becomes $x_1 + \cdots + x_n$ provides an automorphism of $R$ that maps $I$ onto $I$.
\end{proof}

Hence, for monomial algebras, it is enough to decide whether the sum of the variables is a Lefschetz element. As a
consequence,  the presence of the weak Lefschetz property in characteristic zero is
equivalent to the presence of the weak Lefschetz property in some (actually, almost every) positive characteristic.

Recall that a \emph{maximal minor} of a matrix $B$ is the determinant of a maximal square sub-matrix of $B$.

\begin{corollary} \label{lem:wlp-0-p}
    Let $A$ be an  Artinian monomial $K$-algebra, where $K$ is infinite.   Then the following conditions are equivalent:
    \begin{enumerate}
        \item $A$ has the weak Lefschetz property in characteristic zero.
        \item $A$ has the weak Lefschetz property in some positive characteristic.
        \item $A$ has the weak Lefschetz property in every sufficiently large positive characteristic.
    \end{enumerate}
\end{corollary}
\begin{proof}
    Let $\ell = x_0 + \cdots + x_n$. By Proposition~\ref{pro:mono}, $A$ has the weak Lefschetz property if, for each
    integer $d$, the map $\times \ell: [A]_{d-1} \to [A]_d$ has maximal rank. As $A$ is Artinian, there are only
    finitely many non-zero maps to be checked. Fixing monomial bases for all non-trivial components $[A]_j$, the
    mentioned multiplication maps are described by zero-one matrices.

    Suppose $A$ has the weak Lefschetz property in some characteristic $q \geq 0$. Then for each of the finitely many
    matrices above, there exists a maximal minor that is non-zero in $K$, hence non-zero as an integer. The finitely
    many non-zero maximal minors, considered as integers, have finitely many prime divisors. Hence, there are only
    finitely many prime numbers, which divide one of these minors. If the characteristic of $K$ does not belong to this
    set of prime numbers, then $A$ has the weak Lefschetz property.
\end{proof}

We conclude this subsection by noting that any Artinian ideal in two variables has the weak Lefschetz property. This was
first proven for characteristic zero in~\cite[Proposition~4.4]{HMNW} and then for arbitrary characteristic
in~\cite[Corollary~7]{MZ}, though it was not specifically stated therein (see~\cite[Remark~2.6]{LZ}). We provide a
brief, direct argument here.

\begin{proposition} \label{pro:2-wlp}
    Let $R = K[x,y]$, where $K$ is an infinite field of \emph{arbitrary} characteristic. Then every Artinian graded
    algebra $R/I$ has the weak Lefschetz property.
\end{proposition}
\begin{proof}
    Let $\ell \in R$ be a general linear form, and put $s = \min\{j \in \ZZ \st [I]_j \neq 0 \}$. As $[R]_i = [R/I]_i$
    for $i < s$ and multiplication by $\ell$ on $R$ is injective, we see that $[R/I]_{i-1} \rightarrow [R/I]_{i}$ is
    injective if $i < s$. Moreover, since $R/(I,\ell) \cong K[x]/(x^s)$ and $[K[x]/(x^s)]_i = 0$ for $i \geq s$, the map
    $[R/I]_{i-1} \rightarrow [R/I]_{i}$ has a trivial cokernel if $i \geq s$, that is, the map is surjective if $i \geq
    s$. Hence $R/I$ has the weak Lefschetz property.
\end{proof}

Due to this fact, the problem of deciding whether a quotient of $R$ has the weak Lefschetz property is only interesting if $n \ge 3$. In this paper we focus on the case $n=3$, which provides intriguing questions and connections to challenging problems in combinatorics as we are going to show.

\section{Triangular regions} \label{sec:tri}

From now we consider polynomial rings in three variables. For simplicity, we write $R = K[x,y,z]$.
In this section, we begin developing a combinatorial approach for deciding the presence of the  weak Lefschetz property for $R/I$, where $I$ is a monomial ideal. To this end we associate to $I$ a  planar region and consider its tilings by lozenges.

\subsection{Triangular regions and monomial ideals}\label{sub:ideal}~

Let $d \geq 1$ be an integer. Consider an equilateral triangle of side length $d$ that is composed of $\binom{d}{2}$
downward-pointing ($\dntri$) and $\binom{d+1}{2}$ upward-pointing ($\uptri$) equilateral unit triangles. We label the
downward- and upward-pointing unit triangles by the monomials in $[R]_{d-2}$ and $[R]_{d-1}$, respectively, as
follows: place $x^{d-1}$ at the top, $y^{d-1}$ at the bottom-left, and $z^{d-1}$ at the bottom-right, and continue
labeling such that, for each pair of an upward- and a downward-pointing triangle that share an edge, the label of the
upward-pointing triangle is obtained from the label of the downward-pointing triangle by multiplying with a variable.
The resulting labeled triangular region is the \emph{triangular region (of $R$) in degree $d$}
and is denoted $\mathcal{T}_d$. See Figure~\ref{fig:triregion-R}(i) for an illustration.

\begin{figure}[!ht]
    \begin{minipage}[b]{0.48\linewidth}
        \centering
        \includegraphics[scale=1]{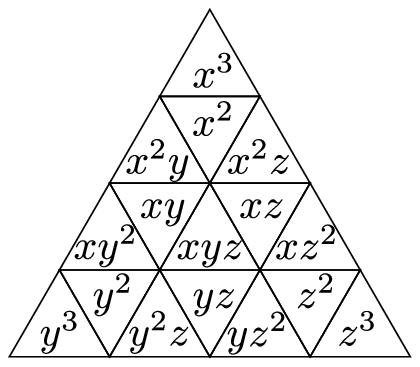}\\
        \emph{(i) $\mathcal{T}_4$}
    \end{minipage}
    \begin{minipage}[b]{0.48\linewidth}
        \centering
        \includegraphics[scale=1]{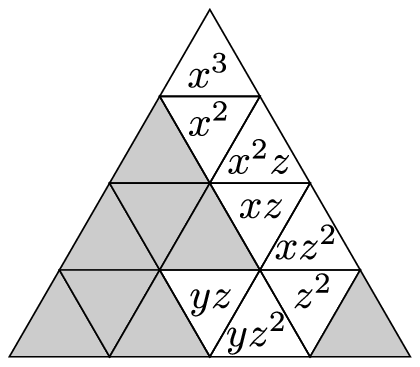}\\
        \emph{(ii) $T_4(xy, y^2, z^3)$}
    \end{minipage}
    \caption{A triangular region with respect to $R$ and with respect to $R/I$.}
    \label{fig:triregion-R}
\end{figure}

Throughout this manuscript we order the monomials of $R$ by using the \emph{graded reverse-lexicographic order}, that is,
$x^a y^b z^c > x^p y^q z^r$ if either $a+b+c > p+q+r$ or $a+b+c = p+q+r$ and the \emph{last} non-zero entry in
$(a-p, b-q, c-r)$ is \emph{negative}. For example, in degree $3$,
\[
    x^3 > x^2y > xy^2 > y^3 > x^2z > xyz > y^2z > xz^2 > yz^2 > z^3.
\]
Thus in $\mathcal{T}_4$, see Figure~\ref{fig:triregion-R}(i), the upward-pointing triangles are ordered starting at
the top and moving down-left in lines parallel to the upper-left edge.

We generalize this construction to quotients by monomial ideals. Let $I$ be any monomial ideal of $R$. The
\emph{triangular region (of $R/I$) in degree $d$}, denoted by $T_d(I)$, is the part of $\mathcal{T}_d$ that is obtained
after removing the triangles labeled by monomials in $I$. Note that the labels of the downward- and
upward-pointing triangles in $T_d(I)$ form $K$-bases of $[R/I]_{d-2}$ and $[R/I]_{d-1}$, respectively. It is sometimes
more convenient to illustrate triangular regions with the removed triangles darkly shaded instead of being removed; both
illustration methods will be used throughout this manuscript. See Figure~\ref{fig:triregion-R}(ii) for an example.

Notice that the regions missing from $\mathcal{T}_d$ in $T_d(I)$ can be viewed as a union of (possibly overlapping)
upward-pointing triangles of various side lengths that include the upward- and downward-pointing triangles inside them.
Each of these upward-pointing triangles corresponds to a minimal generator of $I$ that has, necessarily, degree at most
$d-1$. We can alternatively construct $T_d(I)$ from $\mathcal{T}_d$ by removing, for each minimal generator $x^a y^b
z^c$ of $I$ of degree at most $d-1$, the \emph{puncture associated to $x^a y^b z^c$} which is an upward-pointing
equilateral triangle of side length $d-(a+b+c)$ located $a$ triangles from the bottom, $b$ triangles from the
upper-right edge, and $c$ triangles from the upper-left edge. See Figure~\ref{fig:triregion-punctures} for an example.
We call $d-(a+b+c)$ the \emph{side length of the puncture associated to $x^a y^b z^c$}, regardless of possible overlaps
with other punctures in $T_d (I)$.

\begin{figure}[!ht]
    \begin{minipage}[b]{0.48\linewidth}
        \centering
        \includegraphics[scale=1]{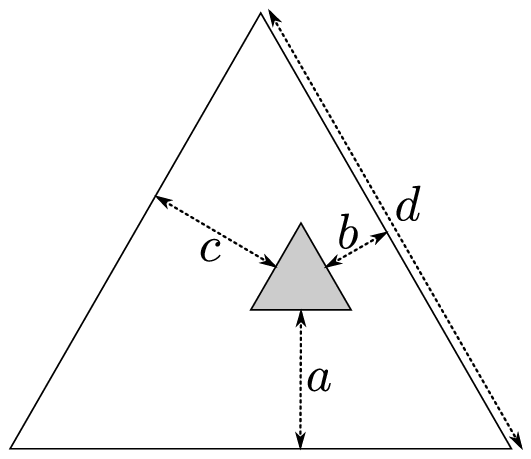}\\
        \emph{(i) $T_{d}(x^a y^b z^c)$}
    \end{minipage}
    \begin{minipage}[b]{0.48\linewidth}
        \centering
        \includegraphics[scale=1]{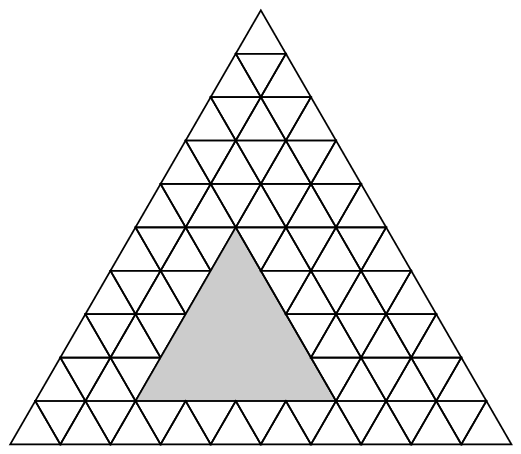}\\
        \emph{(ii) $T_{10}(xy^3z^2)$}
    \end{minipage}
    \caption{$T_d(I)$ as constructed by removing punctures.}
    \label{fig:triregion-punctures}
\end{figure}

We say that two punctures \emph{overlap} if they share at least an edge. Two punctures are said to be \emph{touching}
if they share precisely a vertex.

\subsection{Tilings with lozenges}\label{sub:tiling}~

Now we consider tilings of a triangular region by lozenges. A \emph{lozenge} is a union of two unit equilateral triangles glued together along a shared edge, i.e., a rhombus with
unit side lengths and angles of $60^{\circ}$ and $120^{\circ}$. Lozenges are also called calissons and diamonds in the
literature.

Fix a positive integer $d$ and consider the triangular region $\mathcal{T}_d$ as a union of unit triangles. Thus a \emph{subregion}
$T \subset \mathcal{T}_d$ is a subset of such triangles. We retain their labels. We say that
a subregion $T$ is \emph{$\dntri$-heavy}, \emph{$\uptri$-heavy}, or \emph{balanced} if there are more downward pointing
than upward pointing triangles or less, or if their numbers are the same, respectively. A subregion is \emph{tileable}
if either it is empty or there exists a tiling of the region by lozenges such that every triangle is part of exactly one
lozenge. See Figure~\ref{fig:triregion-intro} for an example.

\begin{figure}[!ht]
    \begin{minipage}[b]{0.48\linewidth}
        \centering
        \includegraphics[scale=1]{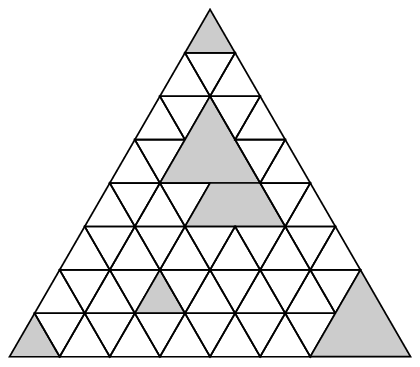}
    \end{minipage}
    \begin{minipage}[b]{0.48\linewidth}
        \centering
        \includegraphics[scale=1]{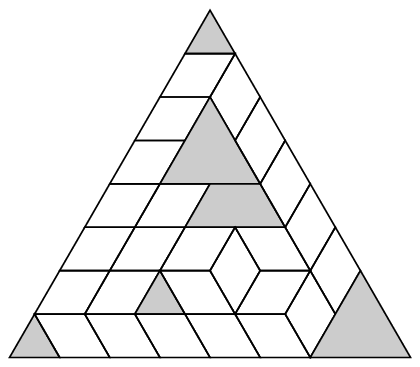}
    \end{minipage}
    \caption{A triangular region $T \subset \mathcal{T}_8$ together with one of its $13$ tilings.}
    \label{fig:triregion-intro}
\end{figure}

Let $T \subset \mathcal{T}_d$ be any subregion. Given a monomial $x^a y^b z^c$ with degree less than $d$, the
\emph{monomial subregion} of $T$ associated to $x^a y^b z^c$ is the part of $T$ contained in the triangle $a$ units from
the bottom edge, $b$ units from the upper-right edge, and $c$ units from the upper-left edge. In other words, this
monomial subregion consists of the triangles that are in $T$ and in the puncture associated to the monomial $x^a y^b z^c$.
See Figure~\ref{fig:triregion-subregion} for an illustration.

\begin{figure}[!ht]
    \includegraphics[scale=1]{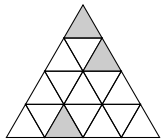}
    \caption{The monomial subregion of $T_{8}(x^7, y^7, z^6, x y^4 z^2, x^3 y z^2, x^4 y z)$
        (see Figure~\ref{fig:triregion-intro}) associated to $x y^2 z$.}
    \label{fig:triregion-subregion}
\end{figure}

Notice that a  tileable subregion is necessarily balanced. In fact, a balanced triangular region $T_d(I)$ is tileable if and only if
    it has no $\dntri$-heavy monomial subregions (see \cite[Theorem 2.2]{CN-resolutions}). The argument uses the following result, which we record for a later application.

\begin{lemma}[{\cite[Lemma 2.1]{CN-resolutions}}] \label{lem:replace-tileable}
    Let $T \subset \mathcal{T}_d$ be any subregion.  If the monomial subregion $U$ of $T$ associated to $x^a y^b z^c$ is tileable,
    then $T$ is tileable if and only if $T \setminus U$ is tileable.

    Moreover, each tiling of $T$ is obtained by combining a tiling of $T \setminus U$ and a tiling of $U$.
\end{lemma}

Let $U \subset \mathcal{T}_d$ be a monomial subregion, and let $T, T' \subset \mathcal{T}_d$ be any subregions such that
$T \setminus U = T' \setminus U$. If $T \cap U$ and $T' \cap U$ are both tileable, then $T$ is tileable if and only if
$T'$ is, by Lemma \ref{lem:replace-tileable}. In other words, replacing a tileable monomial subregion of a triangular
region by a tileable monomial subregion of the same size does not affect tileability.

\section{Enumerations deciding the weak Lefschetz property} \label{sec:enum}

Now we introduce  two enumerations of  \emph{signed} lozenge tilings, using perfect matchings and families of lattice paths, and relate these enumerations to the presence of the weak Lefschetz property.

\subsection{Perfect matchings}\label{sub:pm}

Let $T \subset \mathcal{T}_d$ be any subregion. As above, we consider $T$ as a union of unit triangles. Following \cite{CN-resolutions}, we associate to
$T$ a bipartite graph. This construction  has also been considered in, e.g., \cite{Ci-1997}, \cite{Ci-2005},
and \cite{Ei}.
First, place a vertex at the center of each triangle. Let $B$ be the set of centers of the
downward-pointing triangles, and let $W$ be the set of centers of the upward-pointing triangles.  The \emph{bipartite graph associated to $T$}
is the bipartite graph $G(T)$ on the vertex set $B \cup W$ that has an edge between vertices $B_i \in B$ and $W_j \in W$
if the corresponding upward- and downward-pointing triangle share are edge. In other words, edges of $G(T)$ connect vertices
of adjacent triangles. See Figure~\ref{fig:build-pm}(i).

\begin{figure}[!ht]
    \begin{minipage}[b]{0.32\linewidth}
        \centering
        \includegraphics[scale=1]{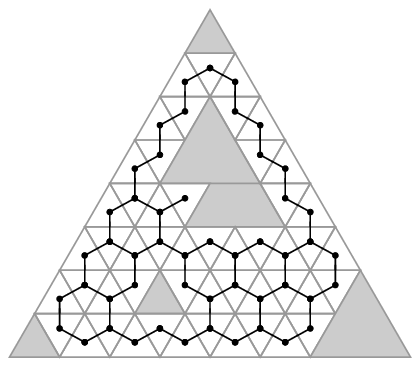}\\
        \emph{(i) The graph $G(T)$.}
    \end{minipage}
    \begin{minipage}[b]{0.32\linewidth}
        \centering
        \includegraphics[scale=1]{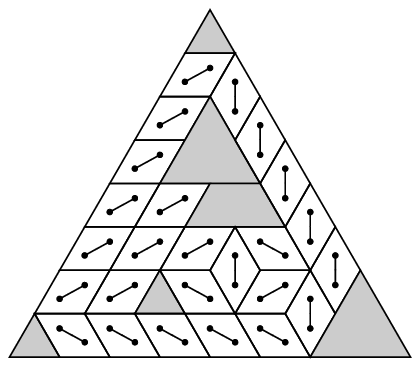}\\
        \emph{(ii) Selected covered edges.}
    \end{minipage}
    \begin{minipage}[b]{0.32\linewidth}
        \centering
        \includegraphics[scale=1]{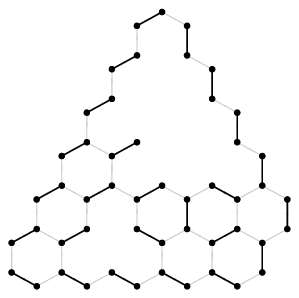}\\
        \emph{(iii) The perfect matching.}
    \end{minipage}
    \caption{Given the tiling $\tau$ in Figure~\ref{fig:triregion-intro}(ii) of $T$, we construct the perfect
        matching $\pi$ of the bipartite graph $G(T)$ associated to $\tau$.}
    \label{fig:build-pm}
\end{figure}

Using the above ordering of the vertices, we define the \emph{bi-adjacency matrix}
of $T$ as the bi-adjacency matrix $Z(T) := Z(G(T))$ of the graph $G(T)$. It is the zero-one matrix $Z(T)$ of
size $\# B \times \# W$ with entries $Z(T)_{(i,j)}$ defined by
\begin{equation*}
    Z(T)_{(i,j)} =
    \begin{cases}
        1 & \text{if $(B_i, W_j)$ is an edge of $G(T)$ } \\
        0 & \text{otherwise.}
    \end{cases}
\end{equation*}
It is a square matrix if and only if the region $T$ is balanced.

A \emph{perfect matching of a graph $G$} is a set of pairwise non-adjacent edges of $G$ such that each vertex is matched.
There is well-known bijection between lozenge
tilings of a balanced subregion $T$ and perfect matchings of $G(T)$. A lozenge tiling $\tau$ is transformed in to a perfect
matching $\pi$ by overlaying the triangular region $T$ on the bipartite graph $G(T)$ and selecting the edges of the graph that
the lozenges of $\tau$ cover. See Figures~\ref{fig:build-pm}.  Using this bijection, it follows that the permanent of the bi-adjacency matrix enumerates the \emph{unsigned} tilings of the region.

\begin{proposition}\label{pro:per}
    If $T \subset \mathcal{T}_d$ is a non-empty balanced subregion, then the lozenge tilings of $T$ are enumerated by $\per{Z(T)}$.
\end{proposition}

Considering a perfect matching $\pi$ as a permutation on $\#\uptri(T) = \#\dntri (T)$ letters, it is natural to assign a
sign to each lozenge tiling using the signature of the permutation $\pi$.

\begin{definition} \label{def:pm-sign}
    Let $T \subset \mathcal{T}_d$ be a non-empty balanced subregion. Then we define the \emph{perfect matching sign}
    of a lozenge tiling $\tau$ of $T$ as $\msgn{\tau} := \sgn{\pi}$, where $\pi \in \PS_{\#\uptri(T)}$ is the perfect
    matching determined by $\tau$.
\end{definition}

Thus  the \emph{perfect matching signed} tilings of a region are enumerated by the determinant of a bi-adjacency matrix.

\begin{theorem}[{\cite[Theorem 3.5]{CN-resolutions}}]\label{thm:pm-det}
     If $T \subset \mathcal{T}_d$ is a non-empty balanced subregion, then the perfect matching signed lozenge tilings of
     $T$ are enumerated by $|\det{Z(T)}|$, that is,
        \[
           \left| \sum_{\tau \text{ tiling of } T}  \msgn{\tau} \right| = |\det Z(T)|.
        \]
\end{theorem}

We recursively define a puncture of $T \subset \mathcal{T}_d$ to be
a \emph{non-floating} puncture if it  touches the boundary of $ \mathcal{T}_d$ or if it overlaps or touches a non-floating puncture of $T$. Otherwise we call a puncture
a \emph{floating} puncture. For example, the region $T$ in Figure \ref{fig:triregion-intro} has three non-floating punctures (in the corners) and three floating punctures, two of them are overlapping and have side length two.

In the case, where the floating punctures all have even side lengths,  the sign $\msgn{\tau}$ is constant.

\begin{proposition}[{\cite[Corollary 4.7]{CN-resolutions}}]\label{pro:same-sign}
    If $T \subset \mathcal{T}_d$ is a tileable triangular region such that all floating punctures of $T$ have an even side length, then
    every lozenge tiling of $T$ has the same perfect matching sign, and so $\per{Z(T)} = |\det{Z(T)}|$.
\end{proposition}

This result applies in particular to tileable, simply connected triangular regions.
\smallskip

We now relate the above results about  triangular regions associated to a monomial ideal $I$ to the weak Lefschetz property of $R/I$. The key is an alternative description of the bi-adjacency matrix $Z(T)$ to $T = T_d (I)$  that involves multiplication by $\ell = x+y+z$.

\begin{proposition}\label{pro:interp-Z}
    Let $I$ be a monomial ideal in $R = K[x,y,z]$, and let $\ell = x+y+z$. Fix an integer $d$ and consider the
    multiplication map $\times(x+y+z): [R/I]_{d-2} \rightarrow [R/I]_{d-1}$. Let $M(d)$ be the matrix to this linear map
    with respect to the monomial bases of $[R/I]_{d-2}$ and $[R/I]_{d-1}$ in reverse-lexicographic order. Then the
    transpose of $M(d)$ is the bi-adjacency matrix $Z(T_d (I))$.
\end{proposition}

\begin{proof}
    Set $s = h_{R/I}(d-2)$ and $t = h_{R/I}(d-1)$, and let $\{m_1, \ldots, m_s\}$ and $\{n_1, \ldots, n_t\}$ be the
    distinct monomials in $[R]_{d-2} \setminus I$ and $[R]_{d-1} \setminus I$, respectively, listed in
    reverse-lexicographic order. Then the matrix $M(d)$ is a $t \times s$ matrix. Its column $j$ is the coordinate
    vector of $\ell m_j = x m_i + y m_i + z m_i$ modulo $I$ with respect to the chosen basis of $[R/I]_{d-1}$. In
    particular, the entry in column $j$ and row $i$ is $1$ if and only if $n_i$ is a multiple of $m_j$.

    Recall  that the rows and columns of $Z(T_d(I))$ are indexed by the downward- and upward-pointing
    unit triangles, respectively. These triangles are labeled by the monomials in $[R]_{d-2} \setminus I$ and
    $[R]_{d-1} \setminus I$, respectively. Since the label of an upward-pointing triangle is a multiple of the label of a
    downward-pointing triangle if and only if the triangles are adjacent, it follows that $Z(T_d(I)) = M(d)^{\rm T}$.
\end{proof}

For ease of reference, we record the following consequence.

\begin{corollary}\label{cor:max-Z}
    Let $I$ be a monomial ideal in $R = K[x,y,z]$.  Then the multiplication map $\times(x+y+z): [R/I]_{d-2} \rightarrow [R/I]_{d-1}$
    has maximal rank if and only if the matrix $Z(T_d(I))$ has maximal rank.
\end{corollary}

Combined with Proposition \ref{pro:mono}, we get a criterion for the presence of the weak Lefschetz property.

\begin{corollary}\label{cor:wlp-biadj}
    Let $I$ be an Artinian monomial ideal in $R = K[x,y,z]$, where $K$ is infinite. Then $R/I$ has the weak Lefschetz property if and only if,
    for each positive integer $d$, the matrix $Z(T_d(I))$ has maximal rank.
\end{corollary}

Assuming large enough socle degrees, it is enough to consider at most two explicit matrices to check for the weak Lefschetz property.

\begin{corollary}\label{cor:wlp-Z}
    Let $I$ be an Artinian monomial ideal in $R = K[x,y,z]$, where $K$ is infinite, and suppose the degrees of the  socle generators of  $R/I$ are at least $d-2$. Then:
    \begin{enumerate}
        \item If $0 \neq h_{R/I}(d-1) = h_{R/I}(d)$, then $R/I$ has the weak Lefschetz property if and only if
            $\det{Z(T_d(I))}$ is not zero in $K$.
            
        \item If $h_{R/I}(d-2) < h_{R/I}(d-1)$ and $h_{R/I}(d-1) > h_{R/I}(d)$, then $R/I$ has the weak Lefschetz
            property if and only if $Z(T_d(I))$ and $Z(T_{d+1}(I))$ both have a maximal minor that is not zero in $K$.
            
    \end{enumerate}
\end{corollary}
\begin{proof}
    By Proposition~\ref{pro:mono}, it is enough to check whether $\ell = x+y+z$ is a Lefschetz element of $R/I$. Hence,
    the result follows by combining Corollary~\ref{cor:max-Z} and Proposition~\ref{pro:wlp} and
    Corollary~\ref{cor:twin-peaks-wlp}, respectively.
\end{proof}

In the case, where the region $T_d(I)$ is balanced, we interpreted the determinant of $Z(T_d(I))$ as the
enumeration of signed perfect matchings on the bipartite graph $G(T_d (I))$. This is also useful for non-balanced regions.

\begin{remark}
   \label{rem:minor as a region}
 For any  region $T_d(I)$, we
interpret a maximal minor of $Z(T_d(I))$ as enumeration for a region obtained  by removing unit triangles from $T_d(I)$, since the rows and columns of
$Z(T_d(I))$ are indexed by the triangles of $T_d(I)$. More precisely, let $T = T_d(I)$ be a $\dntri$-heavy triangular
region with $k$ more downward-pointing triangles than upward-pointing triangles. Abusing notation slightly, we define a
\emph{maximal minor} of $T$ to be a balanced subregion $U$ of $T$ that is obtained by removing $k$ downward-pointing triangles
from $T$. Similarly, if $T$ is $\uptri$-heavy, then we remove only upward-pointing triangles to get a maximal minor.

Clearly, if $U$ is a maximal minor of $T$, then $\det{Z(U)}$ is indeed a maximal minor of $Z(T)$. Thus, $Z(T)$ has
maximal rank if and only if there is a maximal minor $U$ of $T$ such that $Z(U)$ has maximal rank.
\end{remark}

\begin{example} \label{exa:minors-Z}
    Let $I = (x^4, y^4, z^4, x^2 z^2)$. Then the Hilbert function of $R/I$, evaluated between degrees $0$ and $7$, is
    $(1,3,6,10,11,9,6,2)$, and $R/I$ is level with socle degree $7$. Hence, by Corollary~\ref{cor:wlp-Z}, $R/I$ has the
    weak Lefschetz property if and only if $Z(T_5 (I))$ and $Z(T_6 (I))$ both have a maximal minor of maximal rank.

    \begin{figure}[!ht]
        \begin{minipage}[b]{0.54\linewidth}
            \centering
            \includegraphics[scale=1]{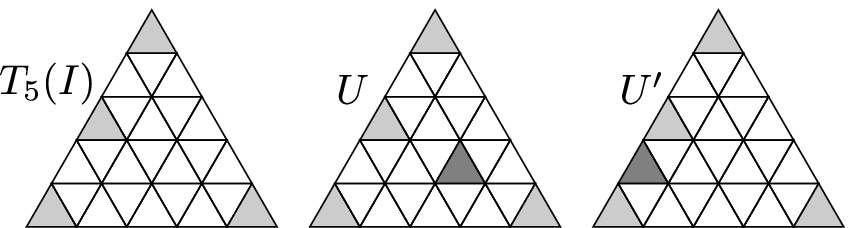}\\
            \emph{(i) $\det{Z(U)} = 0$ and $|\det{Z(U')}| = 4$}
        \end{minipage}
        \begin{minipage}[b]{0.44\linewidth}
            \centering
            \includegraphics[scale=1]{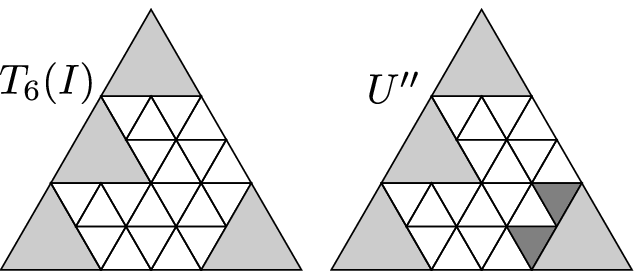}\\
            \emph{(ii) $|\det{Z(U'')}| = 1$}
        \end{minipage}
        \caption{Examples of maximal minors of $T_d(I)$, where $I = (x^4, y^4, z^4, x^2 z^2)$.}
        \label{fig:maximal-minors}
    \end{figure}

    Since $h_{R/I}(3) = 10 < h_{R/I}(4) = 11$, we need to remove $1$ upward-pointing triangle from $T_5(I)$ to get a
    maximal minor of $T_5(I)$; see Figure~\ref{fig:maximal-minors}(i) for a pair of examples. There are $\binom{11}{10} = 11$
    maximal minors, and these have signed enumerations with magnitudes $0, 4,$ and $8$. Thus multiplication from
    degree $3$ to degree $4$ fails injectivity exactly if the characteristic of $K$ is $2$.

    Furthermore, since $h_{R/I}(4) = 11 > h_{R/I}(5) = 9$, we need to remove $2$ downward-pointing triangles from
    $T_6(I)$ to get a maximal minor of $T_6(I)$; see Figure~\ref{fig:maximal-minors}(ii) for an example. There are
    $\binom{11}{9} = 55$ maximal minors, and these have signed enumerations with magnitudes $0, 1,2,3,5,$ and $8$. Thus
    multiplication from degree $4$ to degree $5$ is always surjective (choose the maximal minor whose signed enumeration
    is $1$).

    Hence, we conclude that $R/I$ has the weak Lefschetz property if and only if the characteristic of the base field is
    \emph{not} $2$.
\end{example}

\subsection{Non-intersecting lattice paths}\label{sub:nilp}

Following~\cite[Section~5]{CEKZ} (similarly,~\cite[Section~2]{Fi}) we associate to $T \subset \mathcal{T}_d$ a finite
set $L(T)$ that can be identified with a subset of the lattice $\ZZ^2$.  Abusing notation, we refer to $L(T)$ as a
sub-lattice of $\ZZ^2$. We then employ a relation between lozenge tilings and families of non-intersecting lattice paths (see \cite{CN-resolutions} for details).

To construct $L(T)$, place a vertex at the midpoint of the edge of each triangle of $T$  that is parallel to the upper-left
boundary of the triangle $\mathcal{T}_d$.  We use the monomial label of the upward-pointing triangle on whose   edge the midpoint is located to specify this vertex of $L(T)$.
In order to construct paths on $L(T)$, we think of rightward motion parallel to the
bottom edge of $\mathcal{T}_d$ as ``horizontal'' and downward motion parallel to the upper-right edge of $\mathcal{T}_d$ as
``vertical'' motion. Thus, orthogonalizing $L(T)$ with respect to these motions moves the vertex associated to a monomial $x^a y^b z^{d-1-(a+b)}$ in $L(T)$ to the point $(d-1-b, a)$ in $\mathbb Z^2$ (see Figure~\ref{fig:build-nilp}).

We next single out special vertices of $L(T)$. Label the vertices of $L(T)$ that are only on upward-pointing triangles in $T$, from smallest to largest
in the reverse-lexicographic order, as $A_1, \ldots, A_m$.  Similarly, label the vertices of $L(T)$ that are only on downward-pointing triangles in $T$,
again from smallest to largest in the reverse-lexicographic order, as $E_1, \ldots, E_n$.  See Figure~\ref{fig:build-nilp}(i).  Note that there are an equal
number of vertices  $A_1, \ldots, A_m$ and $E_1, \ldots, E_n$ if and only if the region $T$ is balanced.

A \emph{lattice path} in a lattice $L \subset \ZZ^2$ is a finite sequence of vertices of $L$  so that all single steps move either to the right or down.
Given any vertices $A, E \in \ZZ^2$, the number of lattice paths in $\ZZ^2$ from $A$ to $E$ is a binomial coefficient.  In fact, if $A$ and $E$ have
coordinates $(u,v), (x,y) \in \ZZ^2$, there are $\binom{x-u+v-y}{x-u}$ lattice paths from $A$ to $E$ in $\ZZ^2$.

Using the above identification of $L(T)$ as a sub-lattice of $\ZZ^2$, a \emph{lattice path} in $L(T)$ is a finite sequence of vertices of $L(T)$
so that all single steps move either to the East or to the Southeast. The \emph{lattice path matrix}
of $T$ is the $m \times n$ matrix  $N(T)$ with entries $N(T)_{(i,j)}$ defined by
\[
    N(T)_{(i,j)} = \# \text{lattice paths in $\ZZ^2$ from $A_i$ to $E_j$}.
\]
Thus,  the entries of $N(T)$ are binomial coefficients.

Next we consider several lattice paths simultaneously. A \emph{family of non-intersecting lattice paths}
is a finite collection of lattice paths such that no two lattice paths have any points in common.  If $T$ is balanced, so $m = n$, there is  a well-known bijection between lozenge tilings of $T$ and families of non-intersecting lattice paths from $A_1, \ldots, A_m$ to $E_1, \ldots, E_m$;
see, e.g., the survey~\cite{Pr}.  Let $\tau$ be a lozenge tiling of $T$.  Using the lozenges of $\tau$ as a guide,
we connect each pair of vertices of $L(T)$ that occur on a single lozenge.  This generates the family of non-intersecting lattice
paths $\Lambda$ of $L(T)$ corresponding to $\tau$.  See Figures~\ref{fig:build-nilp}(ii) and~(iii).

\begin{figure}[!ht]
    \begin{minipage}[b]{0.32\linewidth}
        \centering
        \includegraphics[scale=1]{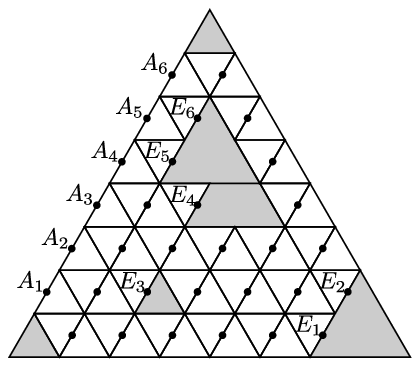}\\
        \emph{(i) The sub-lattice $L(T)$.}
    \end{minipage}
    \begin{minipage}[b]{0.32\linewidth}
        \centering
        \includegraphics[scale=1]{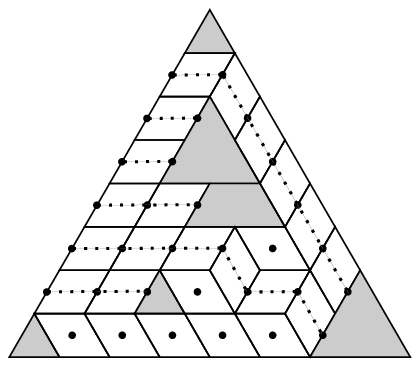}\\
        \emph{(ii) The overlaid image.}
    \end{minipage}
    \begin{minipage}[b]{0.32\linewidth}
        \centering
        \includegraphics[scale=1]{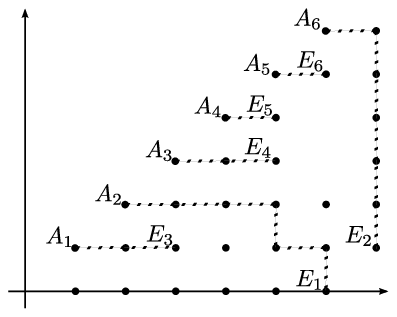}\\
        \emph{(iii) The family $\Lambda$.}
    \end{minipage}
    \caption{The family of non-intersecting lattice paths $\Lambda$ associated to the tiling $\tau$ in Figure~\ref{fig:triregion-intro}(ii).}
    \label{fig:build-nilp}
\end{figure}

Consider now a family $\Lambda$  of $m$ non-intersecting lattice paths in $L(T)$ from
$A_1, \ldots, A_m$ to $E_1, \ldots, E_m$. Then $\Lambda$ determines a permutation $\lambda \in \PS_m$ such that the path in $\Lambda$ that begins at $A_i$ ends at $E_{\lambda(i)}$. Using the signature of $\lambda$ gives another way for assigning a sign to a lozenge tiling of $T$.

\begin{definition} \label{def:nilp-sign}
    Let $T \subset \mathcal{T}_d$ be a non-empty balanced subregion as above, and let $\tau$ be a lozenge tiling of $T$.
    Then we define the \emph{lattice path sign} of $\tau$ as $\lpsgn{\tau} := \sgn{\lambda}$, where $\lambda \in \PS_m$ is
    the permutation such that, for each $i$, the lattice path determined by $\tau$ that starts at $A_i$ ends at $E_{\lambda (i)}$.
\end{definition}

By now we described two signs of a lozenge tiling, one using perfect matchings and one using lattice paths.  In fact, it is shown in \cite{CN-resolutions} that these two signs are essentially the same.
Thus we can enumerate the \emph{signed} tilings of a triangular region $T$ by using its lattice path matrix  $N(T)$ or its bi-adjacency matrix $Z(T)$.

\begin{theorem}[{\cite[Theorems 3.9 and 4.6]{CN-resolutions}}]\label{thm:detZN}
     If $T \subset \mathcal{T}_d$ is a non-empty balanced triangular region, then the perfect matching signed lozenge tilings and the lattice path signed lozenge tilings of
     $T$ are both enumerated by $|\det{Z(T)}| = |\det{N(T)}|$, that is,
        \[
           \left| \sum_{\tau \text{ tiling of } T}  \msgn{\tau} \right| = |\det Z(T)| = |\det{N(T)}| =
           \left| \sum_{\tau \text{ tiling of } T} \lpsgn{\tau} \right|.
        \]
\end{theorem}
\smallskip

We now show that the lattice path matrix $N(T_d(I))$ is also relevant for studying the weak Lefschetz property of $R/I$. If $N(T_d(I))$ is a  square matrix, then this follows from the results in Subsection~\ref{sub:pm} and Theorem \ref{thm:detZN}. However, this matrix is also relevant, even if it is not square. In fact, $N(T_d(I))$ is related to the cokernel of multiplication by $\ell = x+y+z$ on $R/I$.

\begin{proposition} \label{pro:interp-N}
    Let $I$ be a monomial ideal of $R = K[x,y,z]$.   Fix an integer $d$ and set $N = N(T_d(I))$. Then $\dim_K{[R/(I, x+y+z)]_{d-1}} = \dim_K {\ker{N^{\rm T}}}$.
\end{proposition}

\begin{proof}
Note that $R/(I, x+y+z)$ is isomorphic to $[S/J]_{d-1}$,
where $S = K[x,y]$ and $J$ is the ideal generated by the generators of $I$ with $x+y$ substituted for $z$.

We now describe a matrix whose rank equals $\dim_K [J]_{d-1}$. Define an integer $a$ as the least power of $x$ in
    $I$ that is less than $d$, and set $a := d$ if no such power exist. Similarly, define an integer $b \leq d$ using
    powers of $y$ in $I$. Let $G_1$ and $G_3$ be the sets of monomials in $x^a [S]_{d-1-a}$ and $y^b [S]_{d-1-b}$,
    respectively. Furthermore, let $G_2$ be the set consisting of the polynomials $x^p y^{d-1-p-e} (x+y)^e \in
    [J]_{d-1}$ such that $x^i y^j z^e$ is a minimal generator of $I$, where $e > 0,\; i \leq p$, and $j \leq d-1-p-e$.
    Thus, replacing $x+y$ by $z$, each element of $G_2$ corresponds to a monomial $x^p y^{d-1-p-e} z^e \in [I]_{d-1}$. Order
    the elements of $G_2$ by using the reverse-lexicographic order of the corresponding monomials in $[I]_{d-1}$, from
    smallest to largest. Similarly, order the monomials in $G_1$ and $G_3$ reverse-lexicographically, from smallest to
    largest. Note that $G_1 \cup G_2 \cup G_3$ is a generating set for the vector space $[J]_{d-1}$. The coordinate
    vector of a polynomial in $[S]_{d-1}$ with respect to the monomial basis of $[S]_{d-1}$ has as entries the
    coefficients of the monomials in $[S]_{d-1}$. Order this basis again reverse-lexicographically from smallest to
    largest. Now let $M$ be the matrix whose column vectors are the coordinate vectors of the polynomials in $G_1, G_2$,
    and $G_3$, listed in this order. Then $\dim_K [J]_{d-1} = \rank M$ because $G$ generates $[J]_{d-1}$.

    Finally, consider the lattice path matrix $N = N(T_d(I))$. Its rows and columns are indexed by the starting and
    end points of lattice paths, respectively. Fix a starting point $A_i$ and an end point $E_j$. The monomial label of
    $A_i$ is of the form $x^s y^{d-1-s}$, where $x^s y^{d-1-s} \notin I$. Thus, the orthogonalized coordinates of $A_i
    \in \ZZ^2$ are $(s, s)$. The monomial label of the end point $E_j$ is of the form $x^p y^{d-1-p-e} z^e$, where $x^p
    y^{d-1-p-e} z^e$ is a multiple of a minimal generator of $I$ of the form $x^i y^j z^e$. The orthogonalized coordinates of $E_j$ are $(p+e, p)$. Hence there are
    \[
        \binom{p+e-s + s - p}{s-p} = \binom{e}{s-p}
    \]
    lattice paths in $\ZZ^2$ from $A_i$ to $E_j$. By definition, this is the $(i, j)$-entry of  the lattice path matrix $N$.

    The monomial label of the end point $E_j$ corresponds to the polynomial
    \[
        x^p y^{d-1-p-e} (x+y)^e = \sum_{k=0}^e \binom{e}{k} x^{p+k} y^{d-1-p-k}
    \]
    in $G_2$. Thus, its coefficient of the monomial label $x^s y^{d-1-s}$ is $N_{(i, j)}$. It follows that the matrix
    $M$ has the form
    \begin{equation*}
        M =
        \begin{pmatrix}
            {\mathcal I}_{d - b} & * & 0 \\
            0 & N & 0 \\
            0 & * & {\mathcal I}_{d - a}
        \end{pmatrix},
    \end{equation*}
    where we used ${\mathcal I}_{k}$ to denote the $k \times k$ identity matrix.

    Notice that the matrices $M$ and $N$ have $d = \dim_K [S]_{d-1}$ and $a+b - d$ rows, respectively. We conclude that
    \begin{equation*}
        \begin{split}
            \dim_K [S/J]_{d-1} &= d - \dim_K [J]_{d-1}\\
                               &= d - \rank M \\
                               &= a+b-d - \rank N \\
                               &= \dim_K \ker N^{\rm T},
        \end{split}
    \end{equation*}
    as claimed.
\end{proof}

The last result provides another way for checking whether the multiplication by $x+y+z$ has maximal rank.

\begin{corollary} \label{cor:max-N}
    Let $I$ be a monomial ideal in $R = K[x,y,z]$. Then the multiplication map
    $\varphi_d = \times(x+y+z): [R/I]_{d-2} \rightarrow [R/I]_{d-1}$ has maximal rank if and only if $N = N(T_d(I))$
has maximal rank.
\end{corollary}
\begin{proof}
    Consider the exact sequence
    \[
        [R/I]_{d-2} \stackrel{\varphi_d}{\longrightarrow}
        [R/I]_{d-1} \longrightarrow [R/(I, x+y+z)]_{d-1} \longrightarrow 0.
    \]
    It gives that $\varphi_d$ has maximal rank if and only if $ \max\{0, \dim_K{[R/I]_{d-1}} - \dim_K{[R/I]_{d-2}}\} = \dim_K [R/(I, x+y+z)]_{d-1}$.
    By Proposition~\ref{pro:interp-N}, this is equivalent to
    \[
        \dim_K {\ker{N^{\rm T}}} = \max\{0, \dim_K{[R/I]_{d-1}} - \dim_K{[R/I]_{d-2}}\}.
    \]

    Recall that, by construction, the vertices of the lattice $L(T_d(I))$ are on edges of the triangles that are
    parallel to the upper-left edge of ${\mathcal T}_d$, where this edge belongs to just an upward-pointing triangle
    ($A$-vertices), just a downward-pointing triangle ($E$-vertices), or an upward- and a downward-pointing unit
    triangle (all other vertices). Suppose there are $m$ $A$-vertices, $n$ $E$-vertices, and $t$ other vertices. Then
    there are $m + t$ upward-pointing triangles and $n+t$ downward-pointing triangles, that is, $\dim_K{[R/I]_{d-1}} =
    m+t$ and $\dim_K{[R/I]_{d-2}} = n+t$. Hence
    \[
        \dim_K{[R/I]_{d-1}} - \dim_K{[R/I]_{d-2}} = (m+t) - (n+t) = m-n.
    \]
    Since the rows and columns of $N$ are indexed by $A$- and $E$-vertices, respectively, $N$ is an $m \times n$ matrix.
    Hence, $N$ has maximal rank if and only if
    \[
        \dim_K {\ker{N^{\rm T}}} = \max\{0, m-n\} = \max\{0, \dim_K{[R/I]_{d-1}} - \dim_K{[R/I]_{d-2}}\}.
    \]
\end{proof}

Now, using Corollary~\ref{cor:max-N}
instead of Corollary~\ref{cor:max-Z}, we obtain a result that is analogous to Corollary~\ref{cor:wlp-Z}.

\begin{corollary}\label{cor:wlp-N}
    Let $I$ be an Artinian monomial ideal in $R = K[x,y,z]$, where $K$ is infinite, and suppose the degrees of the socle generators of $R/I$
    are at least $d-2$. Then:
    \begin{enumerate}
        \item If $0 \neq h_{R/I}(d-1) = h_{R/I}(d)$, then $R/I$ has the weak Lefschetz property if and only if
            $\det{N(T_d(I))}$ is not zero in $K$.
            
        \item If $h_{R/I}(d-2) < h_{R/I}(d-1)$ and $h_{R/I}(d-1) > h_{R/I}(d)$, then $R/I$ has the weak Lefschetz
            property if and only if $N(T_d(I))$ and $N(T_{d+1}(I))$ both have a maximal minor that is not zero in $K$.
            
    \end{enumerate}
\end{corollary}

In the case where $T = T_d(I)$ is balanced we interpreted the determinant of $N(T)$ as the  enumeration of signed
families of non-intersecting lattice paths in the lattice $L(T)$ (see Theorem \ref{thm:detZN}). In general, we can
similarly interpret the maximal minors of $N(T)$ by removing $A$-vertices or $E$-vertices from $L(T)$, since the rows
and columns of $N(T)$ are indexed by these vertices. Note that removing the $A$- and $E$-vertices is the same as
removing the associated unit triangles in $T$. For example, $U'$ in Figure~\ref{fig:maximal-minors}(i) corresponds to
removing the starting point $A_1$ from $U$. It follows that the maximal minors of $N(T)$ are exactly the determinants of
maximal minors of $T$ (see Remark \ref{rem:minor as a region}) that are obtained from $T$ by removing only unit triangles corresponding to $A$- and $E$-vertices.
We call such a maximal minor a \emph{restricted maximal minor} of $T$.

Clearly, $N(T)$ has maximal rank if and only if there is a restricted maximal minor $U$ of $T$ such that $N(U)$ has
maximal rank. As a consequence, for a $\uptri$-heavy region $T$, it is enough to check the restricted maximal minors in
order to determine whether $Z(T)$ has maximal rank.

\begin{proposition} \label{pro:restricted-only}
    Let $T = T_d(I)$ be an $\uptri$-heavy triangular region. Then $Z(T)$ has maximal rank if and only if there is a
    restricted maximal minor $U$ of $T$ such that $Z(U)$ has maximal rank.
\end{proposition}
\begin{proof}
    By Corollaries~\ref{cor:max-Z} and~\ref{cor:max-N}, we have that $Z(T)$ has maximal rank if and only if $N(T)$ has
    maximal rank. Since each restricted maximal minor $U$ of $T$ is obtained by removing upward-pointing triangles, it
    is the triangular region of some monomial ideal. Thus, Theorem~\ref{thm:detZN} gives $|\det{Z(U)}| = |\det{N(U)}|$.
\end{proof}

\begin{remark} \label{rem:restricted}
    The preceding proposition allows us to reduce the number of minors of $Z(T)$ that need to be considered.
    In Example~\ref{exa:minors-Z}(i), there are $11$ maximal minors of $T_5(I)$, but only $2$ \emph{restricted}
    maximal minors.

    In the special case of a hexagonal region as in Proposition~\ref{pro:ci-enum} below, Proposition~\ref{pro:restricted-only}
    was observed by Li and Zanello in \cite[Theorem~3.2]{LZ}.
\end{remark}

We continue to consider Example~\ref{exa:minors-Z}, using lattice path matrices now.

\begin{example} \label{exa:minors-N}
    Recall the ideal $I = (x^4, y^4, z^4, x^2 z^2)$ from Example~\ref{exa:minors-Z}. By Corollary~\ref{cor:wlp-N}, $R/I$
    has the weak Lefschetz property if and only if $N(T_5(I))$ and $N(T_6(I))$ have maximal rank. Since $N(T_5(I))$ is a
    $2 \times 1$ matrix, we need to remove $1$ $A$-vertex to get a maximal minor (see $U'$ in
    Figure~\ref{fig:maximal-minors}(i) for one of the two choices). Both choices have signed enumeration $4$. Since
    $N(T_6(I))$ is a $0 \times 2$ matrix we need to remove $2$ $E$-vertices to get a restricted maximal minor. The
    region $U''$ in Figure~\ref{fig:maximal-minors}(ii) is the only choice, and the signed enumeration is $1$. Thus, we
    see again that $R/I$ has the weak Lefschetz property if and only if the characteristic of the base field $K$ is not  $2$.
\end{example}~

\section{Explicit enumerations}
    \label{sec:enumeration}

Before applying the methods developed in the previous sections to studying the weak Lefschetz property, we consider the problem of determining  enumerations. We begin by discussing some general techniques. We then use these to evaluate some determinants.

\subsection{Replacements}\label{sub:det-replace}~\par

Recall that, by Lemma~\ref{lem:replace-tileable}, removing a tileable region does not affect unsigned tileability. Using
the structure of the bi-adjacency matrix $Z(T)$, we analyze how removing a balanced region affects signed enumerations.

\begin{proposition} \label{pro:rep-enum}
    Let $T \subset \mathcal{T}_d$   be a balanced  subregion, and let $U$ be a balanced monomial subregion of $T$.
    Then $|\det{Z(T)}| = |\det{Z(T \setminus U)} \cdot \det{Z(U)}|$.
\end{proposition}
\begin{proof}
    Recall that the rows of the matrices $Z(\cdot)$ are indexed by the downward-pointing triangles, and the columns of
    the matrices $Z(\cdot)$ are indexed by the upward-pointing triangles, using the reverse-lexicographic order of their
    monomial labels. Reorder the downward-pointing (respectively, upward-pointing) triangles of $T$ so that the triangles of
    $T \setminus U$ come first and the triangles of $U$ come second, where we preserve the internal order of the triangles
    of $T \setminus U$ and $U$. Using this new ordering, we reorder the rows and columns of $Z(T)$. The result is a block
    matrix
    \[
        \left( \begin{array}{cc} Z(T \setminus U) & X \\ Y & Z(U) \end{array} \right).
    \]
    Since the downward-pointing triangles of $U$ are not adjacent  to any upward-pointing triangle of $T \setminus U$, the
    matrix $Y$ is a zero matrix. Thus,  the claims follow by using the block matrix formula for determinants.
\end{proof}

In particular, if we remove a monomial region with a unique lozenge tiling, then we do not modify the enumerations of
lozenge tilings in that region. This is true in greater generality.

\begin{proposition}\label{pro:remove-unique-tileable}
    Let $T \subset \mathcal{T}_d$ be a balanced subregion, and let $U$ be any subregion of $T$ such that each lozenge
    tiling of $T$ induces a tiling of $U$ and all the induced tilings of $U$ agree. Then we have:
    \begin{enumerate}
        \item $Z(T)$ has maximal rank if and only if $Z(T \setminus U)$ has maximal rank.
        \item $|\det{Z(T)}| = |\det{Z(T \setminus U)}|$.
    \end{enumerate}
\end{proposition}
\begin{proof}
    Part (ii) follows from Theorem \ref{thm:pm-det}, and it implies part (i).
\end{proof}

We point out the following special case.

\begin{corollary}\label{cor:replace-two-punctures}
    Let $T = T_d (I)$ be a balanced triangular region with two punctures $P_1$ and $P_2$ that overlap or touch each
    other. Let $P$ be the minimal covering region of $P_1$ and $P_2$. Then the following statements hold.
    \begin{enumerate}
        \item $\per{Z(T)} = \per{Z(T \setminus P)}$; and
        \item $|\det{Z(T)}| = |\det{Z(T \setminus P)} |$.
    \end{enumerate}
\end{corollary}
\begin{proof}
    The monomial region $U := P \setminus (P_1 \cup P_2)$ is uniquely tileable. Hence the claims follows from
    Proposition~\ref{pro:remove-unique-tileable} because $T \setminus U = T \setminus P$.
\end{proof}

We give an example of such a replacement.

\begin{example} \label{exa:split-puncture}
    Let $T = T_d(I)$ be a balanced triangular region. Suppose the ideal $I$ has minimal generators $x^{a + \alpha} y^b
    z^c$ and $x^a y^{b + \beta} z^{c+\gamma}$. The punctures associated to these generators overlap or touch if and only if $a + \alpha
    + b + \beta + c + \gamma \leq d$. In this case, the minimal overlapping region $U$ of the two punctures is associated to
    the greatest common divisor $x^a y^b z^c$. Assume that $U$ is not overlapped by any other puncture of $T$. Then $U$ is
    uniquely tileable. Hence the regions $T$ and $T' = T \setminus U = T_d (I, x^a y^b z^c)$ have the same enumerations.
    Note that the ideal $(I, x^a y^b z^c)$ has fewer minimal generators than $I$.
\end{example}~

The above procedure allows us in some cases to pass from a triangular region to a triangular region with fewer punctures.
Enumerations are typically more amenable to explicit evaluations if we have few punctures, as we will see in the next
subsection.

\subsection{Determinants}\label{sub:type1det}~

MacMahon~\cite{MacMahon} computed the number of plane partitions (finite two-dimensional arrays that weakly decrease in all columns and rows)
in an $a \times b \times c$ box as (see, e.g., \cite[Page 261]{Pr})
\[
    \Mac(a,b,c) := \frac{\HF(a) \HF(b) \HF(c) \HF(a+b+c)}{\HF(a+b) \HF(a+c) \HF(b+c)},
\]
where $a$, $b$, and $c$ are nonnegative integers and $\HF(n) := \prod_{i=0}^{n-1}i!$ is the \emph{hyperfactorial} of
$n$. David and Tomei proved in~\cite{DT} that plane partitions in an $a \times b \times c$ box are in bijection with
lozenge tilings in a hexagon with side lengths $(a,b,c)$, that is, a hexagon whose opposite sides are parallel and have
lengths $a, b$, and $c$, respectively. However, Propp states on~\cite[Page 258]{Pr} that Klarner was likely the first to
have observed this. See Figure~\ref{fig:pp-tile} for an illustration of the connection.

\begin{figure}[!ht]
    \begin{minipage}[b]{0.48\linewidth}
        \[
            \vspace{1em}
            \begin{array}{cccccc}
               3 & 3 & 2 & 2 & 2 & 1 \\
               3 & 2 & 2 & 1 & 0 & 0
            \end{array}
        \]
    \end{minipage}
    \begin{minipage}[b]{0.48\linewidth}
        \centering
        \includegraphics[scale=1]{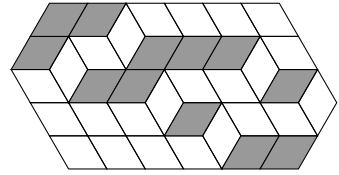}
    \end{minipage}
    \caption{An example of a $2 \times 6 \times 3$ plane partition and the associated lozenge tiling of a hexagon.
        The grey lozenges are the tops of the boxes.}
    \label{fig:pp-tile}
\end{figure}

We use the above formula in several explicit determinantal evaluations. As we are also interested in the prime divisors of
the various non-trivial enumerations we consider, we note that $\Mac(a,b,c) > 0$ and the prime divisors of $\Mac(a,b,c)$
are at most $a+b+c-1$ if $a,b$, and $c$ are positive. This bound is sharp if $a+b+c-1$ is a prime number. If one of $a,
b$, or $c$ is zero, then $\Mac(a,b,c) = 1$.

As a first example, we enumerate the (signed) lozenge tilings of a hexagon.

\begin{proposition} \label{pro:ci-enum}
    Let $a$, $b$, and $c$ be positive integers such that $a \leq b+c$,\; $b \leq a+c$, and $c \leq a+b$. Suppose that
    $d = \frac{1}{2}(a+b+c)$ is an integer. Then $T = T_d(x^a, y^b, z^c)$ is a hexagon with side lengths $(d-a, d-b, d-c)$ and
    \[
        |\det{Z(T)}| = \per{Z(T)} = \Mac(d-a, d-b, d-c).
    \]

    Moreover, the prime divisors of the enumeration are  bounded above by $d-1$.
\end{proposition}
\begin{proof}
    As $a \leq b + c$, we have $d = \frac{1}{2}(a+b+c) \geq \frac{1}{2}(a + a) = a$. Similarly, $d \geq b$ and $d \geq
    c$. Thus $T$ has three punctures of side length $d-a$, $d-b$, and $d-c$ in the three corners. Moreover, $d -(d-a + d-b) =
    d-c$ is the distance between the punctures of length $d-a$ and $d-b$, and similarly for the other two puncture
    pairings.
    Thus, the unit triangles of $T$ form a hexagon with side lengths $(d-a,d-b,d-c)$. By MacMahon's formula we have
    $\per{Z(T)} = \Mac(d-a,d-b,d-c)$. Since $T$ is simply-connected, Proposition \ref{pro:same-sign} gives $|\det{Z(T)}| = \per{Z(T)}  = \Mac(d-a,d-b,d-c)$. The prime divisors of this integer are bounded above by
    $(d-a) + (d-b) + (d-c) - 1 = d-1$.
\end{proof}

Combining Propositions~\ref{pro:rep-enum} and~\ref{pro:ci-enum} we get the enumeration for a slightly more
complicated triangular region. (We will use this observation in Section~\ref{sec:type2}.) Clearly, the process of
removing a hexagon from a puncture can be repeated.

\begin{corollary} \label{cor:ci-nest}
    Let $T = T_d(x^{a+\alpha}, y^b, z^c, x^a y^{\beta}, x^a z^{\gamma})$, where the quadruples $(a,b,c,d)$ and $(\alpha,
    \beta, \gamma, d-a)$ are both as in Proposition~\ref{pro:ci-enum}. In particular, $a + \alpha + \beta + \gamma = b +
    c$ and $d = \frac{1}{2}(a+b+c)$. Then
    \[
        |\det{Z(T)}| = \per{Z(T)} = \Mac(d-a, d-b, d-c) \Mac(d-a-\alpha, d-a-\beta, d-a-\gamma),
    \]
    and the prime divisors of the enumeration are bounded above by $d-1$.
\end{corollary}

\begin{proof}
    The region $T$ is obtained from $T_d(x^a, y^b, z^c)$ by replacing the puncture associated to $x^a$ by
    $T_{d-a}(x^{\alpha}, y^{\beta}, z^{\gamma})$. See Figure~\ref{fig:ci-ci}.  We conclude by using Proposition~\ref{pro:rep-enum}.
    \begin{figure}[!ht]
        \includegraphics[scale=1]{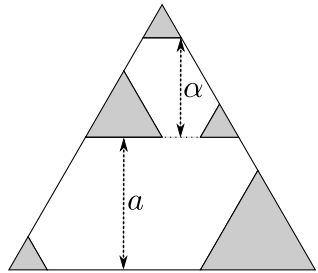}
        \caption{A  hexagon with a puncture replaced by a hexagon.}
        \label{fig:ci-ci}
    \end{figure}
   \end{proof}

As preparation for the next enumeration, we need a more general determinant calculation, which may be of independent interest.

\begin{lemma} \label{lem:split-binom-det}
    Let $M$ be an $n$-by-$n$ matrix  with entries
    \[
        (M)_{i,j} = \left\{
            \begin{array}{ll}
                \displaystyle \binom{p}{q + j - i}     & \mbox{if } 1     \leq j \leq m, \\[1.1em]
                \displaystyle \binom{p}{q + r + j - i} & \mbox{if } m + 1 \leq j \leq n, \\
            \end{array}
        \right.
    \]
    where $p,q,r,$ and $m$ are nonnegative integers and $1 \leq m \leq n$.  Then
    \[
        \det{M} = \Mac(m,q,r) \Mac(n-m, p-q-r, r) \frac{\HF(q+r)\HF(p-q)\HF(n+r)\HF(n+p)}{\HF(n+p-q)\HF(n+q+r)\HF(p)\HF(r)}.
    \]
\end{lemma}
\begin{proof}
    We begin by  using~\cite[Equation~(12.5)]{CEKZ} to evaluate $\det{M}$ to be
    \[
        \prod_{1\leq i < j \leq n} (L_j - L_i) \prod_{i=1}^n \frac{(p+i-1)!}{(n+p-L_i)!(L_i-1)!},
    \]
    where
    \[
        L_j =
        \begin{cases}
            q + j & \text{if } 1 \leq j \leq m, \\
            q + r + j & \text{if } m + 1 \leq j \leq n.
        \end{cases}
    \]
    If we split the products in the previously displayed equation relative to the split in $L_j$, then we obtain the
    following equations:
    \begin{equation*}
        \begin{split}
            \prod_{1\leq i < j \leq n} (L_j - L_i)
                = & \left(\prod_{1\leq i < j \leq m} (j - i)\right) \left(\prod_{m< i < j \leq n} (j - i)\right) \left(\prod_{1\leq i \leq m < j \leq n} (r+j-i)\right) \\[0.3em]
                = & \left(\HF(m)\right) \left(\HF(n-m)\right) \left(\frac{\HF(n+r) \HF(r)}{\HF(n+r-m) \HF(m+r)}\right)
        \end{split}
    \end{equation*}
    and
    \begin{equation*}
        \begin{split}
            \prod_{i=1}^n \frac{(p+i-1)!}{(n+p-L_i)!(L_i-1)!}
                = & \left( \prod_{i=1}^{n}(p+i-1)! \right) \left( \prod_{i=1}^m \frac{1}{(n+p-q-i)!(q+i-1)!}\right) \\[0.3em]
                  & \left( \prod_{i=m+1}^{n} \frac{1}{(n+p-q-r-i)!(q+r+i-1)!}\right) \\[0.3em]
                = & \left( \frac{\HF(n+p)}{\HF(p)} \right) \left( \frac{\HF(n+p-m-q)\HF(q)}{\HF(n+p-q)\HF(m+q)}  \right) \\[0.3em]
                  & \left( \frac{\HF(p-q-r)\HF(m+q+r)}{\HF(n+p-m-q-r)\HF(n+q+r)}  \right).
        \end{split}
    \end{equation*}
    Bringing these equations together we get that $\det{M}$ is
    {\footnotesize \[
        \frac{\HF(m)\HF(q)\HF(r)\HF(m+q+r)}{\HF(m+r)\HF(m+q)}
        \frac{\HF(n-m)\HF(p-q-r)\HF(n+p-m-q)}{\HF(n+r-m)\HF(n+p-m-q-r)}
        \frac{\HF(n+r)\HF(n+p)}{\HF(p)\HF(n+p-q)\HF(n+q+r)},
    \] }
    which, after minor manipulation, yields the claimed result.
\end{proof}

\begin{remark} \label{rem:split-binom-det}
    The preceding lemma generalizes \cite[Lemma~2.2]{LZ}, which handles the case $r = 1$.  Furthermore, if $r = 0$,
    then $\det{M} = \Mac(n, p-q, q)$, as expected (see the running example, $\det \binom{a+b}{a-i+j}$, in~\cite{Kr}).
\end{remark}

We now show that a tileable, simply-connected triangular region with four non-floating punctures has a Mahonian-type determinant. This
particular region is of interest in Section~\ref{sec:type2}. While in the previous evaluations we considered
a bi-adjacency matrix,  we work primarily with a lattice path matrix this time and then use Theorem~\ref{thm:detZN}.

\begin{proposition} \label{pro:two-mahonian}
    Let $T = T_d(x^a, y^b, z^c, x^{\alpha} y^{\beta})$, where $d = \frac{1}{3}(a+b+c+\alpha+\beta)$ is an integer,
    $0 < \alpha < a$, $0 < \beta < b$, and $\max\{a,b,c,\alpha+\beta\} \leq d \leq \min\{a + \beta, \alpha + b, a+c, b+c\}$.
    Then $|\det{Z(T)}| = \per{Z(T)}$ is
    \begin{equation*}
        \begin{split}
                   & \Mac(a+\beta - d,d-a,d-(\alpha+\beta)) \Mac(\alpha+b-d,d-b,d-(\alpha+\beta)) \\[0.3em]
            \times & \frac{\HF(d-a+d-(\alpha+\beta))\HF(d-b+d-(\alpha+\beta))\HF(d-c+d-(\alpha+\beta))\HF(d)}{\HF(a)\HF(b)\HF(c)\HF(d-(\alpha+\beta))}.
        \end{split}
    \end{equation*}
    Moreover, the prime divisors of the enumeration are  bounded above by $d - 1$.
\end{proposition}
\begin{proof}
    Note that $\max\{a,b,c,\alpha+\beta\} \leq d$ implies that all four punctures have nonnegative side length. Furthermore,
    the condition $d \leq \min\{a + \beta, \alpha + b, a+c, b+c\}$ guarantees that none of the punctures overlap. See Figure \ref{fig:two-mahonian}.

    We now compute the lattice path matrix $N(T)$ as introduced in Subsection \ref{sub:nilp}. Recall that a point in the
    lattice $L(T)$ with label $x^u y^v z^{d-1-(u+v)}$ is identified with the point $(d-1-v, u) \in \ZZ^2$. Thus, the
    starting points of the lattice paths are
    \begin{equation*}
        A_i =
        \begin{cases}
            (d-b+i-1, d-b+i-1) & \text{if } 1 \leq i \leq \alpha + b - d, \\
            (2d-(\alpha + \beta+b) + i-1, 2d-(\alpha + \beta+b) + i-1) & \text{if } \alpha + b - d <  i \leq d-c.
        \end{cases}
    \end{equation*}
    For the end points of the lattice paths, we get
    \begin{equation*}
        E_j  = (c-1+j, j-1), \quad \text{where } 1 \leq j \leq d-c.
    \end{equation*}
    Thus, the entries of the lattice path matrix $N(T)$ are
    \[
        (N(T))_{i, j} =
        \left\{
            \begin{array}{ll}
                \displaystyle \binom{c}{d-b+i-j}                    & \mbox{if } 1                \leq i \leq \alpha+b-d, \\[0.8em]
                \displaystyle \binom{c}{2d-(\alpha+\beta+b)+i-j}    & \mbox{if } \alpha+b-d  <  i \leq d-c. \\
            \end{array}
        \right.
    \]
    Transposing $N(T)$, we get a matrix of the form in Lemma~\ref{lem:split-binom-det}, where $m = \alpha + b -d$, $n =
    d-c$, $p = c$, $q = d-b$, and $r = d-(\alpha+\beta)$. Thus, we know $\det N(T)$. Since $T$ is simply-connected,
Proposition \ref{pro:same-sign} and Theorem   \ref{thm:detZN} give $ \per{Z(T)} = |\det{Z(T)}| = \det N(T)$.

    Finally, as $d-\alpha$ and $d-\beta$ are smaller than $d$, the prime divisors of $\det{N(T)} = |\det{Z(T)}|$ are
    bounded above by $d-1$.
\end{proof}

\begin{remark} \label{rem:two-mahonian}
    The evaluation of the determinant in Proposition \ref{pro:two-mahonian} includes two Mahonian terms and a third non-Mahonian
    term. The Mahonian terms can be identified in the triangular region.
    \begin{figure}[!ht]
        \includegraphics[scale=1]{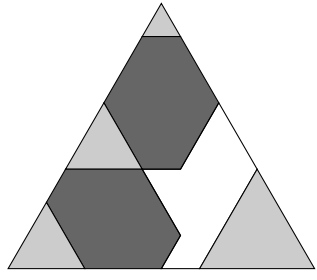}
        \caption{The darkly-shaded hexagons correspond to the two Mahonian terms.}
        \label{fig:two-mahonian}
    \end{figure}
    See Figure~\ref{fig:two-mahonian} where the darkly-shaded hexagons correspond to the Mahonian terms. It is not clear
    (to us) where the third term comes from, though it may be of interest that if one subtracts $d-(\alpha+\beta)$ from
    each hyperfactorial parameter, before the evaluation, then what remains is $\Mac(d-a,d-b,d-c)$.
\end{remark}~

\section{Complete intersections} \label{sec:type1}

In this short section we give a first illustration of our methods by applying them  to the Artinian monomial ideals of $R = K[x,y,z]$ with the fewest number of generators, that is, to the ideals of the form $I = (a^a, y^b, z^c)$. These are
monomial complete intersections and exactly the Artinian monomial ideals of $R$ with type 1. The question whether they have the weak Lefschetz property has motivated a great deal of
research (see \cite{MN-survey} and Remark~\ref{rem:ci-history} below).

Throughout the remainder of this paper we assume that the base field is infinite.
Recall the following result of Reid, Roberts, and Roitman about the shape of Hilbert functions of  complete
intersections.

\begin{lemma}{\cite[Theorem~1]{RRR}} \label{lem:h-ci}
    Let $I = (a^a, y^b, z^c)$, where $a, b,$ and $c$ are positive integers. Then the Hilbert function $h = h_{R/I}$ of
    $R/I$ has the following properties:
    \begin{enumerate}
        \item $h(j-2) < h(j-1)$ if and only if $1 \leq j < \min \{a+b, a+c, b+c, \frac{1}{2}(a+b+c)\}$;
        \item $h(j-2) = h(j-1)$ if and only if $\min \{a+b, a+c, b+c, \frac{1}{2}(a+b+c)\} \leq j \leq \max \{a, b, c, \frac{1}{2}(a+b+c)\}$; and
        \item $h(j-2) > h(j-1)$ if and only if $\max \{a, b, c, \frac{1}{2}(a+b+c)\} < j \leq a+b+c-1$.
    \end{enumerate}
\end{lemma}

Depending on the characteristic of the base field we get the following sufficient conditions that guarantee the weak
Lefschetz property.

\begin{theorem} \label{thm:type-one}
    Let $I = (x^a, y^b, z^c)$, where $a, b,$ and $c$ are positive integers.  Set $d = \flfr{a+b+c}{2}$. Then:
    \begin{enumerate}
        \item If $d < \max\{a,b,c\}$, then $R/I$ has the weak Lefschetz property, regardless of the characteristic of $K$.
        \item If $a+b+c$ is even, then $R/I$ has the weak Lefschetz property in characteristic $p$ if and only if $p$
            does not divide $\Mac(d-a, d-b, d-c)$.
        \item If $a+b+c$ is odd, then $R/I$ has the weak Lefschetz property in characteristic $p$ if and only if $p$
            does not divide any of the integers
            \[
                \frac{\binom{d-1}{a-1}}{\binom{d-1}{i}} \binom{d-c}{a-i-1} \Mac(d-a-1, d-b, d-c),
            \]
            where $d-1-b < i < a$.
    \end{enumerate}

    In any case, $R/I$ has the weak Lefschetz property in characteristic $p$ if $p = 0$ or $p \geq \flfr{a+b+c}{2}$.
\end{theorem}
\begin{proof}
    The algebra $R/I$ has exactly one socle generator. It  has degree $a+b+c-3 \geq d-2$.

    If $d < \max\{a,b,c\}$, then without loss of generality we may assume $a > d$, that is, $a > b+c$. In this case,
    $T_d(I)$ has two punctures, one of length $d-b$ and one of length $d-c$. Moreover, $d-b + d-c = a > d$ so the two
    punctures overlap. Hence $T_d(I)$ is balanced and has a unique tiling. That is, $|\det{Z(T)}| = 1$ and so $R/I$ has
    the weak Lefschetz property, regardless of the characteristic of $K$ (see Corollary~\ref{cor:wlp-Z}).

    Suppose $d \geq \max\{a,b,c\}$.  By Lemma~\ref{lem:h-ci}, we have $h_{R/I}(d-2) \leq h_{R/I}(d-1) > h_{R/I}(d)$.

    Assume $a+b+c$ is even. Then Proposition~\ref{pro:ci-enum} gives that $|\det{Z(T_d(I))}| = \Mac(d-a, d-b, d-c)$, and
    so claim (ii) follows by Corollary~\ref{cor:wlp-Z}.

    Assume $a+b+c$ is odd, and so $d = \frac{1}{2}(a+b+c-1)$. In this case it is enough to find non-trivial maximal
    minors of $T_{d}(I)$ and $T_{d+1}(I)$ by Corollary~\ref{cor:wlp-Z}. Consider the hexagonal regions formed by the
    present unit triangles of each $T_{d+1}(I)$ and $T_{d}(I)$. The former hexagon is obtained from the latter by a
    rotation about $180^{\circ}$. Thus, we need only consider the maximal minors of $T_{d}(I)$. This region has exactly
    one more upward-pointing triangle than downward-pointing triangle. Hence, by Proposition~\ref{pro:restricted-only},
    it suffices to check whether the restricted maximal minors of $T_{d}(I)$ have maximal rank. These minors are exactly
    $T_{i} := T_d(x^a, y^b, z^c, x^{i} y^{d-1-i})$, where $d-1-b < i < a$. Using Proposition~\ref{pro:two-mahonian}, we
    get that $|\det{Z(T_{i})}|$ is
    \[
       \Mac(a -1-i, d - a, 1) \Mac(i + b -d, d-b, 1) \frac{\HF(d-a+1) \HF(d-b+1) \HF(d-c+1) \HF(d)}{\HF(a)\HF(b)\HF(c)\HF(1)},
    \]
    where we notice $d - (i + (d-1-i)) = 1$. Since $\Mac(n, k, 1) = \binom{n+k}{k}$ and $\HF(n) = (n-1)! \HF(n-1)$, for
    positive integers $n$ and $k$, we can rewrite $|\det{Z(T_{i})}|$ as
    \[
        \binom{d-1-i}{d-a} \binom{i}{d-b} \frac{(d-b)! (d-c)!}{(a-1)!} \Mac(d-a-1, d-b, d-c).
    \]
    Simplifying this expression, we get part (iii).

    Finally, using both Propositions~\ref{pro:ci-enum} and~\ref{pro:two-mahonian} we see that the prime divisors of
    $|\det{Z(T_{i})}|$ are bounded above by $d-1$ in each case.
\end{proof}

As announced, we briefly comment on the history of the last result and the research it motivated.

\begin{remark} \label{rem:ci-history}
    The presence of the weak Lefschetz property for monomial complete intersections has been studied by many authors.
    The fact that \emph{all} monomial complete intersections, in any number of variables, have the strong Lefschetz
    property in characteristic zero was proven first by Stanley~\cite{Stanley-1980} using the Hard Lefschetz Theorem.
    (See \cite{Co}, and the references contained therein, for more on the history of this theorem.) However, the weak
    Lefschetz property can fail in positive characteristic.

    The weak Lefschetz property in arbitrary characteristic in the case where one generator has much larger degree than
    the others (case (i) in the preceding proposition) was first established by Watanabe~\cite[Corollary~2]{Wa} for
    arbitrary complete intersections in three variables, not just monomial ones. Migliore and
    Mir\'o-Roig~\cite[Proposition~5.2]{MM} generalized this to complete intersections in $n$ variables.

    Part (ii) of the above result was first established by the authors~\cite[Theorem~4.3]{CN-IJM} (with an extra
    generator of sufficiently large degree), and independently by Li and Zanello~\cite[Theorem~3.2]{LZ}. The latter also
    proved part (iii) above (use $i = a - k$). However, while both papers mentioned the connection to lozenge tilings of
    hexagons, it was Chen, Guo, Jin, and Li~\cite{CGJL} who provided the first combinatorial explanation. In particular,
    the case (ii) was studied in~\cite[Theorem~1.2]{CGJL}. We also note that~\cite[Theorem~4.3]{LZ} can be recovered
    from Theorem~\ref{thm:type-one} if we set $a = \beta + \gamma$, $b = \alpha + \gamma$, and $c = \alpha + \beta$.

    More explicit results have been found in the special case where all generators have the same degree, i.e., $I_a =
    (x^a, y^a, z^a)$. Brenner and Kaid used the idea of a syzygy gap to explicitly classify the prime characteristics in
    which $I_a$ has the weak Lefschetz property~\cite[Theorem~2.6]{BK-p}. Kustin, Rahmati, and Vraciu used this result
    in \cite{KRV}, in which they related the presence of the weak Lefschetz property of $R/I_a$ to the finiteness of the
    projective dimension of $I_a : (x^n + y^n + z^n)$. Moreover, Kustin and Vraciu later gave an alternate explicit
    classification of the prime characteristics in which $I_a$ has the weak Lefschetz property~\cite[Theorem~4.3]{KV}.

    As a final note, Kustin and Vraciu~\cite{KV} also gave an explicit classification of the prime characteristics in
    which monomial complete intersections in arbitrarily many variables with all generators of the same degree have the
    weak Lefschetz property. This was expanded by the first author~\cite[Theorem~7.2]{Co} to an explicit classification
    of the prime characteristics, in which the algebra has the \emph{strong Lefschetz property}. In this work another
    combinatorial connection was used to study the presence of the weak Lefschetz property for monomial complete
    intersections in arbitrarily many variables.
\end{remark}

\section{Type $2$ monomial ideals} \label{sec:type2}

Boij, Migliore, Mir\'o-Roig, Zanello, and the second author proved in \cite[Theorem~6.2]{BMMNZ} that the Artinian
monomial algebras of type two in three variables that are \emph{level} have the weak Lefschetz property in
characteristic zero. The proof given there is surprisingly intricate and lengthy. In this section, we establish a more
general result using techniques derived in the previous sections.

To begin, we describe the Artinian monomial ideals $I$ in $R = K[x,y,z]$ such that $R/I$ has type two, that is, its socle is
of the form $\soc (R/I) \cong K(-s) \oplus K(-t)$. The algebra $R/I$ is level if the socle degrees $s$ and $t$ are
equal. The classification in the level case has been established in \cite[Proposition~6.1]{BMMNZ}. The following more
general result is obtained similarly.

\begin{lemma} \label{lem:classify-type-two}
    Let $I$ be an Artinian monomial ideal in $R = K[x,y,z]$ such that $R/I$ is of type $2$. Then, up to a change of
    variables, $I$ has one of the following two forms:
    \begin{enumerate}
        \item $I = (x^a, y^b, z^c, x^{\alpha} y^{\beta})$, where $0 < \alpha < a$ and $0 < \beta < b$. In this case, the
            socle degrees of $R/I$ are $a + \beta + c-3$ and $\alpha + b + c-3$. Thus, $I$ is level if and only if
            $a - \alpha = b - \beta$.
        \item $I = (x^a, y^b, z^c, x^{\alpha} y^{\beta}, x^{\alpha} z^{\gamma})$, where $0 < \alpha < a$, $0 < \beta < b$,
            and $0 < \gamma < c$. In this case, the socle degrees of $R/I$ are $a + \beta + \gamma-3$ and $\alpha + b + c-3$.
            Thus, $I$ is level if and only if $a - \alpha = b - \beta + c - \gamma$.
    \end{enumerate}
\end{lemma}
\begin{proof}
    We use Macaulay-Matlis duality. An Artinian monomial algebra of type two over $R$ arises as the inverse system of
    two monomials, say $x^{a_1} y^{b_1} z^{c_1}$ and $x^{a_2} y^{b_2} z^{c_2}$, such that one does not divide the other.
    Thus we may assume without loss of generality that $a_1 > a_2$ and $b_1 < b_2$. We consider two cases: $c_1 = c_2$
    and $c_1 \neq c_2$.

    Suppose first that $c_1 = c_2$.  Then the annihilator of the monomials is the ideal
    \[
        (x^{a_1+1}, y^{b_1+1}, z^{c_1+1}) \cap (x^{a_2+1}, y^{b_2+1}, z^{c_1+1}) =   (x^{a_1+1}, y^{b_2+1}, z^{c_1+1}, x^{a_2+1} y^{b_1+1}),
    \]
    which is the form in (i).  By construction, the socle elements are $x^{a_1} y^{b_1} z^{c_1}$ and $x^{a_2} y^{b_2} z^{c_1}$.

    Now suppose $c_1 \neq c_2$; without loss of generality we may assume $c_1 < c_2$. Then the annihilator of the
    monomials is the ideal
    \[
        (x^{a_1+1}, y^{b_1+1}, z^{c_1+1}) \cap (x^{a_2+1}, y^{b_2+1}, z^{c_2+1}) =
        (x^{a_1 + 1}, y^{b_2+1}, z^{c_2+1}, x^{a_2 + 1} y^{b_1+1}, x^{a_2+1} z^{c_1+1}),
    \]
    which is the form in (ii).  By construction, the socle elements are $x^{a_1} y^{b_1} z^{c_1}$ and $x^{a_2} y^{b_2} z^{c_2}$.
\end{proof}

We now give a complete classification of the type two algebras that have the weak Lefschetz property in characteristic zero.

\begin{theorem} \label{thm:type-two}
    Let $I$ be an Artinian monomial ideal in $R = K[x,y,z]$, where $K$ is a field of characteristic zero, such that
    $R/I$ is of type $2$. Then $R/I$ fails to have the weak Lefschetz property in characteristic zero if and only if
    $I = (x^a, y^b, z^c, x^{\alpha} y^{\beta}, x^{\alpha} z^{\gamma})$, up to a change of variables, where
    $0 < \alpha < a$, $0 < \beta < b$, and $0 < \gamma < c$, and there exists an integer $d$ with
    \begin{equation} \label{eqn:type-two}
        \begin{split}
                \max \left\{a, \alpha + \beta, \alpha + \gamma, \frac{a+\alpha+\beta+\gamma}{2} \right\} < d \hspace*{7cm}\\
                < \min \left\{a+\beta + \gamma, \frac{\alpha+b+c}{2}, b+c, \alpha + c, \alpha + b \right\}.
        \end{split}
    \end{equation}
\end{theorem}
\begin{proof}
    According to Corollary \ref{cor:wlp-biadj}, for each integer $d > 0$, we have to decide whether the bi-adjacency
    matrix $Z(T_d(I))$ has maximal rank. This is always true if $d = 1$. Let $d \geq 2$.

    By Lemma~\ref{lem:classify-type-two}, we may assume that $I$ has one of two forms given there. The difference
    between the two forms is an extra generator, $x^{\alpha} z^{\gamma}$. In order to determine the rank of $Z( T_d(I))$
    we split $T = T_d(I)$ across the horizontal line $\alpha$ units from the bottom edge. We call the monomial
    subregion above the line, which is the subregion associated to $x^\alpha$, the \emph{upper portion} of $T$, denoted
    by $T^u$, and we call the isosceles trapezoid below the line the \emph{lower portion} of $T$, denoted by $T^l$. Note
    that $T^u$ is empty if $d \leq \alpha$. Both portions, $T^u$ and $T^l$, are hexagons, i.e., triangular regions
    associated to complete intersections. In particular, if $I$ has four generators, then
    $T^u = T_{d-\alpha}(x^{a-\alpha}, y^{\beta}, z^c)$. Similarly, if $I$ has five generators, then
    $T^u = T_{d-\alpha}(x^{a-\alpha}, y^{\beta}, z^{\gamma})$. In both cases $T^l$ is $T_d(x^{\alpha}, y^b, z^c)$. See
    Figure~\ref{fig:decompose} for an illustration of this decomposition.
    \begin{figure}[!ht]
        \includegraphics[scale=2]{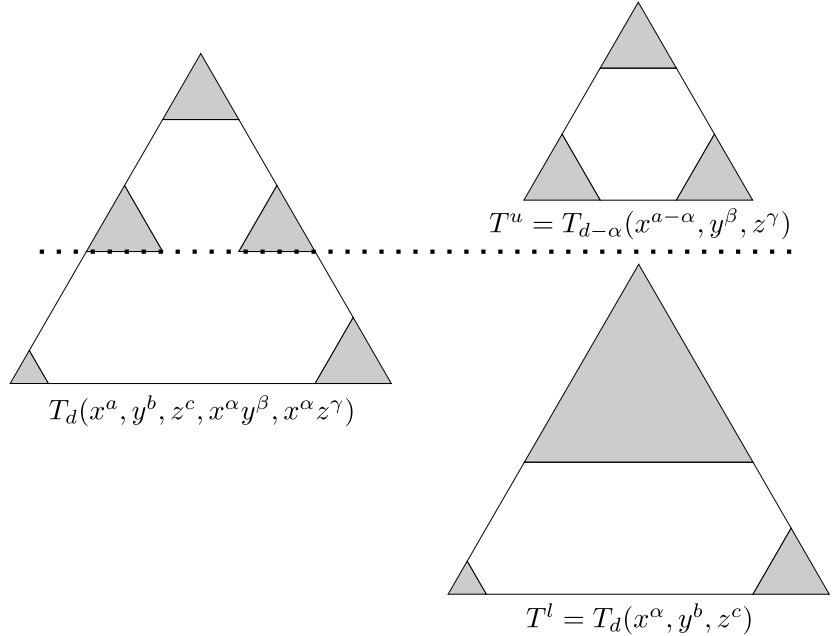}
        \caption{The decomposition of $T_d(I)$ into $T^u$ and $T^l$.}
        \label{fig:decompose}
    \end{figure}

    After reordering rows and columns of the bi-adjacency matrix  $Z(T)$, it becomes a  block matrix of the form
    \begin{equation}\label{eq:block-Z}
        Z =
        \begin{pmatrix}
            Z(T^u) & 0 \\
            Y & Z(T^l)
        \end{pmatrix}
    \end{equation}
    because the downward-pointing triangles in $T^u$ are not adjacent to any upward-pointing triangles in $T^l$. For
    determining when $Z$ has maximal rank, we study several cases, depending on whether $T^u$ and $T^l$ are $\uptri$-heavy,
    balanced, or $\dntri$-heavy.
    \smallskip

    First, suppose one of the following conditions is satisfied: (i) $T^u$ or $T^l$ is balanced, (ii) $T^u$ and $T^l$
    are both $\uptri$-heavy, or (iii) $T^u$ and $T^l$ are both $\dntri$-heavy. In other words, $T^u$ and $T^l$ do not
    ``favor'' triangles of opposite orientations. Since $T^u$ and $T^l$ are triangular regions associated to complete
    intersections, both $Z(T^u)$ and $Z(T^l)$ have maximal rank by  Theorem~\ref{thm:type-one}. Combining non-vanishing
    maximal minors of $Z(T^u)$ and $Z(T^l)$, if follows that the matrix $Z$ has maximal rank as well.
    \smallskip

    Second, suppose $T^u$ is $\uptri$-heavy and $T^l$ is $\dntri$-heavy. We will show that $Z$ has maximal rank in this case.

    Let $t_u = \#\uptri(T^u) - \#\dntri(T^u)$ and $t_l = \#\dntri(T^l) - \#\uptri(T^l)$ be the number of excess
    triangles of each region. In a first step, we show that we may assume $t_u = t_l$. To this end we remove enough of
    the appropriately oriented triangles from the more unbalanced of $T^u$ and $T^l$ until both regions are equally
    unbalanced. Set $t = \min\{t_u, t_l\}$.

    Assume $T^u$ is more unbalanced, i.e., $t_u > t$. Since $T^u$ is $\uptri$-heavy, the top $t_u$ rows of ${\mathcal
    T}_d$ below the puncture associated to $x^a$ do not have a puncture. Thus, we can remove the top $t_u - t$
    upward-pointing triangles in $T^u$ along the upper-left edge of ${\mathcal T}_d$, starting at the puncture
    associated to $x^a$, if present, or in the top corner otherwise. Denote the resulting subregion of $T$ by $T'$.
    Notice that $Z$ has maximal rank if $Z(T')$ has maximal rank. Furthermore, the $t_u - t$ rows in which $T$ and $T'$
    differ are uniquely tileable. Denote this subregion of $T'$ by $U$ (see Figure~\ref{fig:type-2-case-8}(i) for an
    illustration). By construction, the upper and the lower portion ${T^u}'$ and ${T^l}' = T^l$, respectively, of $T'
    \setminus U$ are equally unbalanced. Moreover, $Z(T')$ has maximal rank if and only if $Z(T' \setminus U)$ has
    maximal rank by Proposition~\ref{pro:remove-unique-tileable}. As desired, $T$ and $T' \setminus U$ have the same shape.

    Assume now that $T^l$ is more unbalanced, i.e., $t_l > t$. Since $T^l$ is $\dntri$-heavy, the two punctures
    associated to $x^b$ and $x^c$, respectively, cover part of the bottom $t_l$ rows of ${\mathcal T}_d$. Thus, we can
    remove the bottom $t_l - t$ downward-pointing triangles of $T^l$ along the puncture associated to $x^c$. Denote the
    resulting subregion of $T$ by $T'$. Notice that $Z$ has maximal rank if $Z(T')$ has maximal rank. Again, the $t_l -
    t$ rows in which $T$ and $T'$ differ form a uniquely tileable subregion. Denote it by $U$. By construction, the
    upper and the lower portion ${T^u}' = T^u$ and ${T^l}' $, respectively, of $T' \setminus U$ are equally
    unbalanced. Moreover, $Z(T')$ has maximal rank if and only if $Z(T' \setminus U)$ has maximal rank by
    Proposition~\ref{pro:remove-unique-tileable}. As before, $T$ and $T' \setminus U$ have the same shape.

    The above discussion shows it is enough to prove that the matrix $Z$ has maximal rank if $t_u = t_l = t$, i.e., $T$
    is balanced. Since $T$ has no floating punctures, Proposition~\ref{pro:same-sign} gives the desired maximal rank of
    $Z$ once we know that $T$ has a tiling. To see that $T$ is tileable, we first place $t$ lozenges across the line
    separating $T^u$ from $T^l$, starting with the left-most such lozenge.  Indeed, this is possible since $T^u$ has $t$
    more upwards-pointing than downwards-pointing triangles. Next, place all fixed lozenges. The portion of ${T^u}$ that
    remains untiled after placing these lozenges is a hexagon. Hence it is tileable.
    (See Figure~\ref{fig:type-2-case-8}(ii) for an illustration.)  
    
    \begin{figure}[!ht]
        \begin{minipage}[b]{0.48\linewidth}
            \centering
            \includegraphics[scale=1]{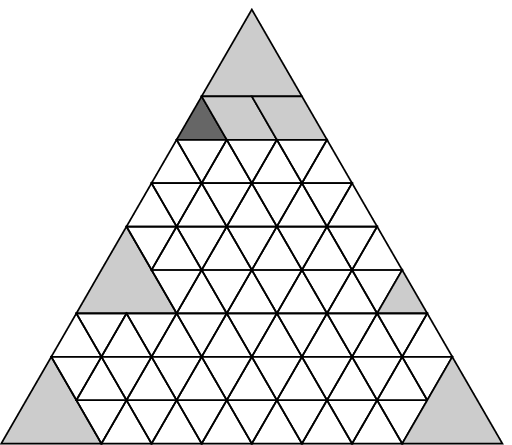}\\
            \emph{(i) A maximal minor of $T$; the removed triangle is darkly shaded.}
        \end{minipage}
        \begin{minipage}[b]{0.48\linewidth}
            \centering
            \includegraphics[scale=1]{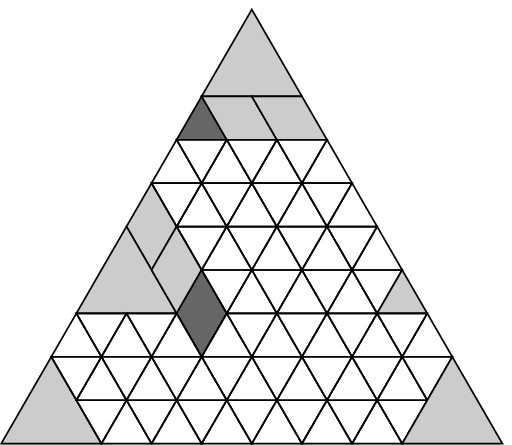}\\
            \emph{(ii) Placing a lozenge on the maximal minor to produce a tiling.}
        \end{minipage}
        \caption{Let $T = T_{10}(x^8, y^8, z^8, x^3 y^5, x^3 z^6)$.
            The lightly-shaded lozenges are fixed lozenges.}
        \label{fig:type-2-case-8}
    \end{figure}
    
    Consider now the portion of ${T^l}$ that remains untiled after placing these lozenges. Since $t$ is at most the number of 
    horizontal rows of $T^l$ this portion is, after a $60^{\circ}$ rotation, a region as described in Proposition~\ref{pro:two-mahonian}.
    Thus it is tileable. Figure~\ref{fig:type-2-case-8-overlay} illustrates this procedure with an example.

    \begin{figure}[!ht]
        \includegraphics[scale=1]{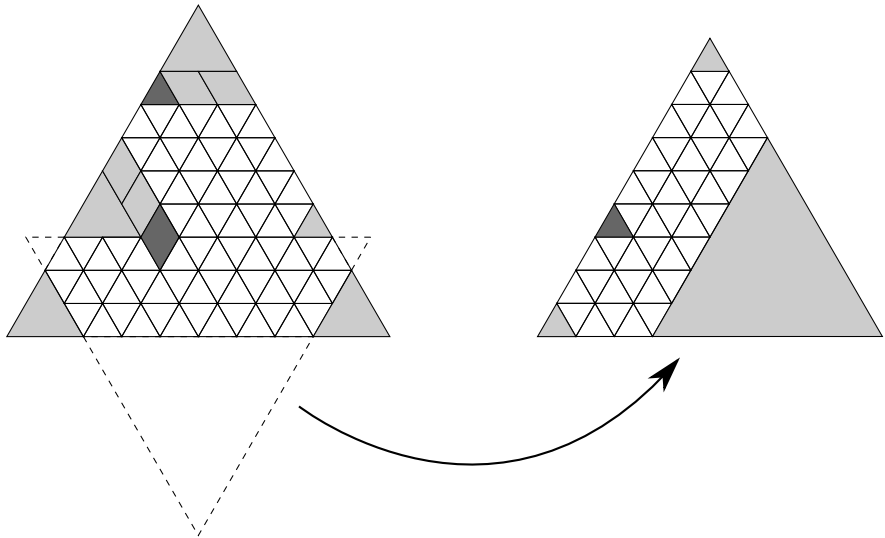}
        \caption{With $T$ as in Figure~\ref{fig:type-2-case-8}, 
            after a $60^{\circ}$ rotation $T^l$ becomes a previously described region.}
        \label{fig:type-2-case-8-overlay}
    \end{figure}

    It follows that $T$ is tileable. Therefore $Z$ has maximal rank, as desired.
    \smallskip

    Finally, suppose $T^u$ is $\dntri$-heavy and $T^l$ is $\uptri$-heavy. Consider any maximal minor of $Z(T)$. It
    corresponds to a balanced subregion $T'$ of $T$. Then its upper portion ${T^u}'$ is still $\dntri$-heavy, and its
    lower portion ${T^l}'$ is $\uptri$-heavy. Hence, any covering of ${T^u}'$ by lozenges must also cover some
    upward-pointing triangles of ${T^l}'$. The remaining part of ${T^l}'$ is even more unbalanced than ${T^l}'$. This
    shows that $T'$ is not tileable. Thus, $\det Z(T') = 0$ by Theorem \ref{thm:pm-det}. It follows that $Z$ does not
    have maximal rank in this case.
    \smallskip

    The above case analysis proves that $R/I$ fails the weak Lefschetz property if and only if there is an integer $d$
    so that the associated regions $T^u$ and $T^l$ are $\dntri$-heavy and $\uptri$-heavy, respectively. It remains to
    determine when this happens.

    If $I$ has only four generators, then no row of $T^u$ has more downward-pointing than upward-pointing triangles.
    Hence, $T^u$ is not $\dntri$-heavy. It follows that $I$ must have five generators if $R/I$ fails to have the weak
    Lefschetz property. For such an ideal $I$, the region $T^u = T_{d-\alpha}(x^{a-\alpha}, y^{\beta}, z^{\gamma})$ is
    $\dntri$-heavy if and only if
    \[
        \dim_K [R/(x^{a-\alpha}, y^{\beta}, z^{\gamma})]_{d-\alpha-2} > \dim_K [R/(x^{a-\alpha}, y^{\beta}, z^{\gamma})]_{d-\alpha-1},
    \]
    and $T^l = T_d(x^{\alpha}, y^b, z^c)$ is $\uptri$-heavy if and only if
    \[
        \dim_K [R/(x^{\alpha}, y^b, z^c)]_{d-2} < \dim_K [R/(x^{\alpha}, y^b, z^c)]_{d-1}.
    \]
    Using Lemma~\ref{lem:h-ci}, a straight-forward computation shows that these two inequalities are both true if and
    only of $d$ satisfies Condition~\eqref{eqn:type-two}.
\end{proof}

\begin{remark}\label{rem:type2}
    The above argument establishes the following more precise version of Theorem~\ref{thm:type-two}:

    Let $R/I$ be a Artinian monomial algebra of type 2, where $K$ is a field of characteristic zero, and let $\ell \in
    R$ be a general linear form. Then the multiplication map $\times \ell: [R/I]_{d-2} \to [R/I]_{d-1}$ does not have
    maximal rank if and only if $I = (x^a, y^b, z^c, x^{\alpha} y^{\beta}, x^{\alpha} z^{\gamma})$, up to a change of
    variables, and $d$ satisfies Condition~\eqref{eqn:type-two}.
\end{remark}

Condition~\eqref{eqn:type-two} in Theorem~\ref{thm:type-two} is indeed non-vacuous.

\begin{example}
    We provide three examples, the latter two come from \cite[Example~6.10]{BMMNZ}, with various shapes of Hilbert functions.
    \begin{enumerate}
        \item Let $I = (x^4, y^4, z^4, x^3y, x^3z)$. Then $d = 5$ satisfies Condition~\eqref{eqn:type-two}.
            Moreover, $T_d(I)$ is a balanced region, and $R/I$ has a strictly unimodal Hilbert function, \\
            $(1,3,6,10,10,9,6,3,1)$.
        \item Let $J = (x^3, y^7, z^7, xy^2, xz^2)$. Then Condition~\eqref{eqn:type-two} is satisfied if and only if $d=5$ or  $d = 6$.
            Note that $R/J$ has a non-unimodal Hilbert function, \\
            $(1,3,6,7,6,6,7,6,5,4,3,2,1)$.
        \item Let $J' = (x^2, y^4, z^4, xy, xz)$. Then $d = 3$ satisfies Condition~\eqref{eqn:type-two}.
            Moreover, $J'$ has a \emph{non-strict} unimodal Hilbert function $(1,3,3,4,3,2,1)$.
    \end{enumerate}
\end{example}

Using Theorem \ref{thm:type-two}, we easily recover \cite[Theorem~6.2]{BMMNZ}, one of the main results in the recent
memoir~\cite{BMMNZ}.

\begin{corollary} \label{cor:type-two-level}
    Let $R/I$ be a Artinian monomial algebra of type 2 over a field of characteristic zero. Then $R/I$ has the weak
    Lefschetz property.
\end{corollary}
\begin{proof}
    By Theorem~\ref{thm:type-two}, we know that if $I$ has four generators, then $R/I$ has the weak Lefschetz property.
    If $I$ has five generators, then it suffices to show that Condition~\eqref{eqn:type-two} is vacuous in this case.
    Indeed, since $R/I$ is level, we have that $a - \alpha = b - \beta + c - \gamma$ by
    Lemma~\ref{lem:classify-type-two}. This implies
    \[
        \frac{a+\alpha+\beta+\gamma}{2} = \frac{2 \alpha+b+c}{2} \geq \alpha + \min\{b, c\}.
    \]
    Hence, no integer  $d$ satisfies Condition~\eqref{eqn:type-two}.
\end{proof}

Moreover, in most of the cases when the weak Lefschetz property holds in characteristic zero, we can give a linear lower
bound on the characteristics for which the weak Lefschetz property must hold. 

\begin{corollary}\label{cor:type-two-pos-char}
    Let $R/I$ be a Artinian monomial algebra of type 2. Suppose that $R/I$ has the weak Lefschetz property in
    characteristic zero and that there is no integer $d$ such that
    \begin{equation}\label{eq:Cond-on-d}
        \max\left\{ \alpha, b, c, \frac{\alpha + b + c}{2} \right\}
        < d <
        \min\left\{ a + \beta, a + \gamma, \alpha + \beta + c, \frac{a + \alpha + \beta + c}{2} \right\}.
    \end{equation}
    Then $R/I$ has the weak Lefschetz property, provided $K$ has  characteristic $p \geq \flfr{\alpha + b + c}{2}$.
\end{corollary}
\begin{proof}
    We use the notation introduced in the proof of Theorem~\ref{thm:type-two}. Fix any integer $d \ge2$. Recall that,
    possibly after reordering rows and columns, the bi-adjacency matrix of $T = T_d (I)$ has the form (see
    Equation~\eqref{eq:block-Z})
    \[
        Z =
        \begin{pmatrix}
            Z(T^u) & 0 \\
            Y & Z(T^l)
        \end{pmatrix}.
    \]
    By assumption, $d$ does not satisfy Condition~\eqref{eqn:type-two} nor \eqref{eq:Cond-on-d}. This implies that $T$
    has one of the following properties: (i) $T^u$ or $T^l$ is balanced, (ii) $T^u$ and $T^l$ are both $\uptri$-heavy,
    or (iii) $T^u$ and $T^l$ are both $\dntri$-heavy.

    The matrices $Z(T^u)$ and $Z(T^l)$ have maximal rank by Theorem~\ref{thm:type-one} if the characteristic of $K$ is
    at least $\flfr{a - \alpha + \beta + c}{2}$ and $\flfr{\alpha + b + c}{2}$, respectively. Combining non-vanishing
    maximal minors of $Z(T^u)$ and $Z(T^l)$, if follows that the matrix $Z$ has maximal rank as well if $\charf K \geq \flfr{\alpha + b + c}{2}$.
\end{proof}

In order to fully extend Theorem~\ref{thm:type-two} to sufficiently large positive characteristics, it remains to consider the
case where $T^u$ is $\uptri$-heavy and $T^l$ is $\dntri$-heavy. This is more delicate.

\begin{example}
    Let $T = T_{10}(x^8, y^8, z^8, x^3 y^5, x^3 z^6)$ as in Figure~\ref{fig:type-2-case-8}, and let $T'$ be the maximal
    minor given in Figure~\ref{fig:type-2-case-8}(i). In each lozenge tiling of $T'$, there is exactly one lozenge that
    crosses the splitting line. There are four possible locations for this lozenge; one of these is illustrated in
    Figure~\ref{fig:type-2-case-8}(ii). The enumeration of lozenge tilings of $T'$ is thus the sum of the lozenge
    tilings with the lozenge in each of the four places along the splitting line. Each of the summands is the product of
    the enumerations of the resulting upper and lower regions. In particular, we have that
    \begin{equation*}
        \begin{split}
            |\det{N(T')}| &= 20 \cdot 60 + 45 \cdot 64 + 60 \cdot 60 + 50 \cdot 48 \\
                          &= 2^4 \cdot 3 \cdot 5^2 + 2^6 \cdot 3^2 \cdot 5 + 2^4 \cdot 3^2 \cdot 5^2 + 2^5 \cdot 3 \cdot 5^2 \\
                          &= 2^5 \cdot 3^2 \cdot 5 \cdot 7 \\
                          &= 10080.
        \end{split}
    \end{equation*}
    Notice that while the four summands only have prime factors of $2$, $3$, and $5$, the final enumeration also has a
    prime factor of $7$.
\end{example}

Still, we can give a bound in this case, though we expect that it is very conservative. It provides the following
extension of Theorem~\ref{thm:type-two}.

\begin{proposition}\label{pro:type2-pos-char}
    Let $R/I$ be a Artinian monomial algebra of type 2 such that $R/I$ has the weak Lefschetz property in characteristic
    zero. Then $R/I$ has the weak Lefschetz property in positive characteristic, provided
    $\charf K \geq 3^e$, where $e = \frac{1}{2}\binom{\frac{1}{2} (a+b+c) + 2}{2}$.
\end{proposition}

This follows from Lemma~\ref{lem:classify-type-two} and the following general result, which provides an effective bound
for Corollary~\ref{lem:wlp-0-p} in the case of three variables.

\begin{proposition}\label{pro:char-0-to-p}
    Let $R/I$ be any Artinian monomial algebra such that $R/I$ has the weak Lefschetz property in characteristic zero.
    If $I$ contains the powers $x^a, y^b, z^c$, then $R/I$ has the weak Lefschetz property in positive characteristic
    whenever $\charf K > 3^{\frac{1}{2}\binom{\frac{1}{2} (a+b+c) + 2}{2}}$.
\end{proposition}
\begin{proof}
    Define $I' = (x^a, y^b, z^c)$, and let $d'$ be the smallest integer such that $0 \neq h_{R/I'}(d'-1) \geq h_{R/I'}(d')$.
    Thus, $d' - 1 \leq \frac{1}{2}(a+b+c)$ by Lemma~\ref{lem:h-ci}.

    Let $d$ be the smallest integer such that $0 \neq h_{R/I}(d-1) \geq h_{R/I}(d)$. Then $d \leq d'$, as $I' \subset I$
    and adding or enlarging punctures only exacerbates the difference in the number of upward- and downward-pointing
    triangles. Since $R/I$ has the weak Lefschetz property in characteristic zero, the Hilbert function of $R/I$ is
    strictly increasing up to degree $d-1$. Hence, Proposition~\ref{pro:wlp} implies that the degrees of non-trivial
    socle elements of $R/I$ are at least $d-1$. The socle of $R/I$ is independent of the characteristic of $K$.
    Therefore Proposition~\ref{pro:wlp} shows that, in any characteristic, $R/I$ has the weak Lefschetz property if and
    only if the bi-adjacency matrices of $T_d (I)$ and $T_{d+1} (I)$ have maximal rank. Each row and column of a
    bi-adjacency matrix has at most three entries that equal one. All other entries are zero. Moreover the maximal
    square sub-matrices of $Z(T_d (I))$ and $Z(T_{d+1} (I))$ have at most $h_{R/I} (d-1)$ rows. Since
    $h_{R/I} (d-1) < h_R (d-1) = \binom{d+1}{2} \leq 3e$, Hadamard's inequality shows
    that the absolute values of the maximal minors of $Z(T_d (I))$ and $Z(T_{d+1} (I))$, considered as integers, are
    less than $3^{2e}$. Hence, any prime number $p \geq 3^e$ does not divide any of these non-trivial maximal minors.
\end{proof}

As indicated above, we believe that the bound in Proposition~\ref{pro:type2-pos-char} is far from being optimal.
Through a great deal of computer experimentation, we offer the following conjecture.

\begin{conjecture} \label{con:type-two-pos-char}
    Let $I$ be an Artinian monomial ideal in $R = K[x,y,z]$ such that $R/I$ is of type two. If $R/I$ has the weak
    Lefschetz property in characteristic zero, then $R/I$ also has the weak Lefschetz property in characteristics
    $p > \frac{1}{2}(a+b+c)$.
\end{conjecture}



\end{document}